\documentclass[11pt,a4paper]{amsart}
\usepackage[a4paper,centering,left=2.2cm,right=2.2cm]{geometry}
\usepackage{amsfonts,amscd,amssymb,amsmath,amsthm,hyperref,mathrsfs,multirow,xcolor,lscape,longtable,pbox,lipsum}
\usepackage[thinlines]{easytable}
\usepackage{microtype}
\usepackage{color}

\newtheorem{thm}{Theorem}

\newtheorem{lema}[thm]{Lemma}

\newtheorem{coro}[thm]{Corollary}

\newtheorem*{thm*}{Theorem}

\newcommand{\Q}{\mathbb{Q}}

\newcommand{\R}{\mathbb{R}}

\newcommand{\s}{\mathbb{S}}

\newcommand{\eqr}[1]{\mbox{(\ref{eq:#1})}}
\newcommand{\ie}{i.e.\ }

\newcommand{\C}{\mathbb{C}}
\newcommand{\V}{\widetilde{V}}

\begin{document}

\title[Ghost Classes]{Ghost classes in $\mathbb{Q}$-rank two orthogonal Shimura varieties} 
\date{\today}
\author{Jitendra Bajpai}
\address{Mathematisches Institut, Georg-August Universit\"at G\"ottingen, Bunsenstrasse 3-5, D-37073 Germany}
\email{jitendra@math.uni-goettingen.de}

\author{Matias V. Moya Giusti} 
\address{IHES, 35 Route de Chartres, 91440 Bures-sur-Yvette, France}
\email{moya@ihes.fr}

\begin{abstract}

In this article, the existence of ghost classes for the Shimura varieties associated to algebraic groups of orthogonal similitudes of signature $(2, n)$ is investigated. We make use of the study of the weights in the mixed Hodge structures associated to the corresponding cohomology spaces and results on the Eisenstein cohomology of the algebraic group of orthogonal similitudes of signature $(1, n-1)$. For the values of $n = 4, 5$ we prove the non-existence of ghost classes for most of the irreducible representations (including most of those with an irregular highest weight). For the rest of the cases, we prove strong restrictions on the possible weights in the space of ghost classes and, in particular,  we show that they satisfy the weak middle weight property.  
\end{abstract}

\maketitle
%\tableofcontents
  
\section{Introduction}\label{intro} 
Let $(\mathrm{G}, X)$ be a Shimura pair, and let $\rho:\mathrm{G}  \rightarrow \mathrm{GL}(V)$ be an irreducible finite dimensional representation (not necessarily defined over $\mathbb{Q}$). For every open compact subgroup $K_f \subset \mathrm{G}(\mathbb{A}_f)$ of the group of finite adelic points of $\mathrm{G}$, we consider the level variety
\[
    S_K = \mathrm{G}(\mathbb{Q}) \backslash X \times (\mathrm{G}(\mathbb{A}_f) / K_f)
\]
and we denote by $S$ the projective limit, over the directed set of open compact subgroups, of these level varieties (i.e. the space of complex points of the corresponding Shimura variety).
One can define in a natural way a local system $\V$ on the Shimura variety $S$ associated to $(\mathrm{G}, X)$, underlying a variation of complex Hodge structure of a given weight $wt(V)$.

Let $\mathrm{A} \subset \mathrm{G}$ be a maximal $\Q$-split torus and $\mathrm{T} \subset \mathrm{G}$ a maximal torus defined over $\Q$ such that $\mathrm{A} \subset \mathrm{T}$. We choose  systems of positive roots in the corresponding root systems $\Phi(\mathrm{G}, \mathrm{T}), \Phi(\mathrm{G}, \mathrm{A})$ so that they are compatible, i.e. the restriction to $\mathrm{A}$, of a positive root in $\Phi(\mathrm{G}, \mathrm{T})$ is either zero or positive in $\Phi(\mathrm{G}, \mathrm{A})$. Let $\lambda : \mathrm{T}(\C) \longrightarrow \C^{\times}$ be the highest weight of $V$. We will usually denote $V$ by $V_\lambda$. The choice of the system of positive roots $\Phi^+(\mathrm{G}, \mathrm{A})$ in $\Phi(\mathrm{G}, \mathrm{A})$ defines a set of standard proper $\Q$-parabolic subgroups denoted by $\mathcal{P}_\Q(\mathrm{G})$.

From now on we will assume that the semisimple $\Q$-rank of $\mathrm{G}$ is $2$. In this case $\mathcal{P}_\Q(\mathrm{G})$ consists of three elements: two maximal $\mathbb{Q}$-parabolic subgroups denoted by $\mathrm{P}_1$ and $\mathrm{P}_2$, and a minimal $\mathbb{Q}$-parabolic subgroup denoted by $\mathrm{P}_0$.

We consider the Borel-Serre compactification $\overline{S}$ of $S$ (see \cite{BoSe73}). The inclusion $S \hookrightarrow \overline{S}$ is a homotopy equivalence and $\V_\lambda$ can be extended naturally to $\overline{S}$.  The corresponding local system will again be denoted by $\V_\lambda$. In fact there is a natural isomorphism $H^\bullet(S, \V_\lambda) \cong H^\bullet(\overline{S}, \V_\lambda)$ and as a consequence we obtain a long exact sequence in cohomology
\begin{equation} \label{eq:first}
\ldots \rightarrow H_c^q(S, \V_\lambda) \rightarrow H^q(S, \V_\lambda) \xrightarrow{r^q} H^q(\partial \overline{S}, \V_\lambda) \rightarrow \ldots
\end{equation}  
where $H_c^\bullet(S, \V_\lambda)$ denotes the cohomology with compact support and $\partial \overline{S} = \overline{S} - S$ is the boundary of the Borel-Serre compactification.

On the other hand, we have a covering $\partial \overline{S} = \cup_{\mathrm{P} \in \mathcal{P}_\mathbb{Q}(\mathrm{G})}\partial_{\mathrm{P}}$, where this union is indexed by the elements of $\mathcal{P}_\mathbb{Q}(\mathrm{G})$. The aforementioned covering induces a spectral sequence in cohomology abutting to $H^\bullet(\partial \overline{S}, \V_\lambda)$ and in the case of $\mathbb{Q}$-rank $2$ this is just a long exact sequence in cohomology
\begin{equation} \label{eq:second}
\ldots \rightarrow H^q(\partial \overline{S}, \V_\lambda) \xrightarrow{p^q} H^q(\partial_{\mathrm{P}_1}, \V_\lambda) \oplus H^q(\partial_{\mathrm{P}_2}, \V_\lambda) \rightarrow H^q(\partial_{\mathrm{P}_0}, \V_\lambda) \rightarrow \ldots
\end{equation}  

We define the space of $q$-ghost classes by $Gh^q(\V_\lambda) = Im(r^q) \cap Ker(p^q)$. Both long exact sequences in cohomology~\eqr{first} and~\eqr{second} are long exact sequences of mixed Hodge structures (see \cite{HaZu-II-94}).

For each $i \in \left\{0, 1, 2\right\}$ there is a decomposition (see \cite{Sch90}, Section $7.2$):
\begin{equation} \label{eq:decomposition}
H^q(\partial_{\mathrm{P}_i}, \V_\lambda) = \bigoplus_{w \in \mathcal{W}^{\mathrm{P}_i}} Ind_{\mathrm{P}_i(\mathbb{A}_f)}^{\mathrm{G}(\mathbb{A}_f)} H^{q-\ell(w)}(S^{\mathrm{M}_i}, \widetilde{W}_{w_\ast(\lambda)})
\end{equation} 
obtained by using Kostant's theorem (see~\cite{Kostant61}), where the induction is the algebraic (unnormalized) induction, $\mathcal{W}^{\mathrm{P}_i}$ is the set of Weyl representatives associated to $\mathrm{P}_i$, $S^{\mathrm{M}_i}$ is the symmetric space associated to the Levi quotient $\mathrm{M}_i$ of $\mathrm{P}_i$, $\ell(w)$ denotes the length of the element $w$ and $W_{w_\ast(\lambda)}$ is the irreducible representation of $\mathrm{M}_i$ with highest weight $w_\ast(\lambda)$ (see Section ~\ref{Decomposition} for the definition of $w_\ast(\lambda)$).

The mixed Hodge structure on $H^q(\partial_{\mathrm{P}_i}, \V_\lambda)$ splits completely with respect to the aforementioned decomposition (see Remark $5.5.6$ of \cite{HaZu-II-94}). Moreover, for $i \in\{1, 2\}$ there exists a subset $\mathcal{W}_i^0 \subset \mathcal{W}(\mathrm{G}, \mathrm{T})$ such that $\mathcal{W}_i^0 \mathcal{W}^{\mathrm{P}_i} = \mathcal{W}^{\mathrm{P}_0}$ and the corresponding  morphism in cohomology $r_i:H^\bullet(\partial_{\mathrm{P}_i}, \V_\lambda) \rightarrow H^\bullet(\partial_{\mathrm{P}_0}, \V_\lambda)$ restricted to $Ind_{\mathrm{P}_i(\mathbb{A}_f)}^{\mathrm{G}(\mathbb{A}_f)} H^{q-\ell(w)}(S^{\mathrm{M}_i}, \widetilde{W}_{w_\ast(\lambda)})$ (with $w \in \mathcal{W}^{\mathrm{P}_i}$) has image in 
\[
\bigoplus_{\widetilde{w} \in \mathcal{W}_i^0} Ind_{_{\mathrm{P}_0(\mathbb{A}_f)}}^{^{\mathrm{G}(\mathbb{A}_f)}} H^{q-\ell(w)-\ell(\widetilde{w})}(S^{\mathrm{M}_0}, \widetilde{W}_{(\widetilde{w}w)_\ast(\lambda)}).
\]

In the cases to be studied in this article, $S^{\mathrm{M}_0}$ has non trivial cohomology only in degree zero, and when $i=0$, the mixed Hodge structure of each term in~\eqref{eq:decomposition} has a pure weight. 

The idea behind this paper is the fact that the space $Ker(p^q)$ (involved in the definition of ghost classes) is the image of the connecting homomorphism $ H^{q-1}(\partial_{\mathrm{P}_0}, \V_\lambda) \rightarrow H^q(\partial \overline{S}, \V_\lambda)$ from the long exact sequence \eqref{eq:second} and, after \eqref{eq:decomposition}, we have a list of possible weights in the corresponding space of ghost classes. By using mixed Hodge theory and Eisenstein cohomology, a study of the morphisms $r^\bullet:H^\bullet(S, \V_\lambda) \rightarrow H^\bullet(\partial \overline{S}, \V_\lambda),$  $r_i:H^\bullet(\partial_{\mathrm{P}_i}, \V_\lambda) \rightarrow H^\bullet(\partial_{\mathrm{P}_0}, \V_\lambda)$ is used to rule out most of the possible weights in the space of ghost classes.

Ghost classes were introduced by A. Borel~\cite{Borel84} in 1984.  Later on, G. Harder mentioned these classes several times in his work.  At the very end in the article~\cite{Harder90}, Harder refers to the case of $\mathrm{GL}_3$ and said \emph{ ``... the ghost classes appear if some $L$-values vanish. The order of vanishing does not play a role. But this may change in the higher rank case"}. 
He further added that this aspect is worth investigating. Not to mention much, this gives a nice motivation to pursue the study of ghost classes further and specially in higher rank groups. Since then, though some mathematicians have studied them, the general theory of these classes has been slow in coming. 
\par

Ghost classes can be introduced for any reductive algebraic group and their definition does not depend on the existence of a complex structure. In the case of a Shimura variety, the space of ghost classes is equipped with  mixed Hodge structure. It is then interesting to study the nontriviality of the space of ghost classes for a Shimura variety and to give some description of the possible weights in its mixed Hodge structure. When $S$ is a Shimura variety, the local system $\V_\lambda$ defines a complex variation of Hodge structure of a certain weight $wt(V_\lambda)$  (see \cite{Zu81} for this notion) and it is known that the weights in the mixed Hodge structure on the space $H^q(S, \V_\lambda)$ are greater than or equal to the middle weight, given by $q+wt(V_\lambda)$ (see Theorem $2.2.7$ of \cite{Harris94}). Therefore the weights in the mixed Hodge structure on the space of ghost classes are greater than or equal to the middle weight. We say that the Shimura variety satisfies the \emph{weak middle weight property} if for every finite dimensional highest weight representation $V_\lambda$ of $\mathrm{G}$ and every nonnegative integer $q$, the only possible weights in the mixed Hodge structure on the space of $q$-ghost classes, in $H^{q}(\partial \overline{S}, \V_\lambda)$, are the middle weight and the middle weight plus one. In addition, the Shimura variety is said to satisfy the \emph{middle weight property} if, for every choice of highest weight $\lambda$ and every nonnegative integer $q$, the only possible weight in the space of ghost classes in $H^q(\partial \overline{S}, \V_\lambda)$ is the middle weight. 

The middle weight property is expected to be true by the experts, but there is no proof of this fact for the moment. We were unable to trace down the attribution of this conjecture in the literature and therefore we consider it a \emph{folklore conjecture}. Recently, the second author has provided a strong support for the (weak) middle weight property by a thorough study of the cases of the Shimura varieties associated to $\mathrm{GSp}(4)$ in~\cite{MM1} and $\mathrm{GU}(2,2)$ in~\cite{MM2}.

In this article, we present a study of the Shimura varieties associated to the groups of orthogonal similitudes $\mathrm{GO}(2, n)$ for $n \geq 3$. The study of ghost classes is discussed in detail, in the last two sections, for the cases $n=4$ and $n=5$. For example, in the case of $n=5$ (see Theorem ~\ref{Thm1}) we obtain the following result: 

\begin{thm*}
Let $V_\lambda$ be the finite dimensional irreducible representation with highest weight $\lambda = a_1 \epsilon_1 + a_2 \epsilon_2 + a_3 \epsilon_3 + c \kappa$. One has:\vspace{-0.1cm}
\begin{enumerate}
\item If $a_2 \neq 0$ then there are no ghost classes in the cohomology space $H^\bullet(\partial \overline{S}, \V_\lambda)$.\\
\vspace{-0.2cm}
\item If $a_2 = 0$ (which implies $a_3=0$ and therefore in terms of fundamental weights one has $\lambda = a_1 \varpi_1 + c \kappa$), then the only possible weights in the mixed Hodge structure of the space of ghost classes are the middle weight and the middle weight plus one.
\end{enumerate}
\end{thm*}

\par We obtain a similar result in the case $n=4$ (see Theorem~\ref{Thm2}). When the highest weight $\lambda$ of the irreducible representation is regular, one can obtain the non-existence of ghost classes by combining Theorem 4.11,~\cite{Sch83} and Theorem 19,~\cite{FRANKE1981}. In this article, we take a step further and prove the non-existence of ghost classes for most of the irregular highest weights. In the remaining cases we restrict the list of degrees in cohomology in which there could exist ghost classes and prove that there is, in each degree, only one possible weight in their corresponding mixed Hodge structure which is in all cases the middle weight or the middle weight plus one.  

\section{The Shimura Variety Involved}\label{tsv}

In this section we present the family of Shimura varieties to be studied. Throughout the article, $n \geq 3$. We denote by $\mathbb{G}_m$ the multiplicative group and by $\mathbb{S}$ the restriction of scalars $Res_{\mathbb{C}/\mathbb{R}}\mathbb{G}_m$. That is, 
\begin{equation}
    \mathbb{S}(F) = \mathbb{G}_m(F \otimes_\mathbb{R} \mathbb{C}) \quad \mbox{ for every } \mathbb{R} \mbox{-algebra } F. \nonumber
\end{equation}
and in particular $\mathbb{S}(\mathbb{R}) = \mathbb{C}^\times$ is the multiplicative group of $\mathbb{C}$. We denote by $z, \bar{z}:\mathbb{S}(\mathbb{C}) \rightarrow \mathbb{C}^\times$ the algebraic characters of $\mathbb{S}(\mathbb{C})$ such that the composition of $\mathbb{C}^\times = \mathbb{S}(\mathbb{R}) \hookrightarrow \mathbb{S}(\mathbb{C})$ with them are respectively the identity and the complex conjugation. 
Consider the Shimura pair $(\mathrm{GO}(2, n), X)$, where $\mathrm{GO}(2, n)$ is the group of orthogonal similitudes of signature $(2,n)$ defined by 
\[ \mathrm{GO}(2,n)(A) =  \left\{ g \in \mathrm{GL}_{n+2}(A) \mid g^t I_{2, n} g = \nu(g) I_{2, n}, \nu(g) \in A^\times \right\}, \]
for every $\Q \mbox{-algebra } A$, where $I_{2, n} = -2Id_2 \times Id_{n-2} \times 2Id_2$ and $X$ is the $\mathrm{GO}(2, n)(\R)$-conjugacy class of homomorphisms containing the element $h:\s(\mathbb{R}) \longrightarrow \mathrm{GO}(2, n)(\R)$ given by
\[ 
h(x+iy) = \left[ \begin{array}{cccccc}
x^2 - y^2 & 2xy & & & &\\
-2xy & x^2 - y^2 & & & & \\
 & & x^2 + y^2 & & & \\
 & & & \ddots & & \\
 & & & & x^2 + y^2 & \\
 & & & & & x^2 + y^2 \end{array} \right] \quad \forall (x+iy) \in \mathbb{S}(\mathbb{R}).
 \]

Thus, the weight morphism $\omega_{_X}: \mathbb{G}_m \longrightarrow \mathrm{GO}(2, n)$ of the Shimura pair is given by $\omega_{_X}(t) = t^2 Id_{n+2}$ where $Id_{n+2}$ denotes the identity in dimension $n+2$.

The choice of $I_{2, n}$ may seem a bit artificial, but we are using it only to get the description of $h$ and being able to work with the more canonical quadratic form defined below by $J_n$. In fact, what follows is also valid for general orthogonal groups of signature $(2, n)$ but we will keep working with this particular orthogonal group in order to give an explicit description of this case.

For the description of the parabolic subgroups it is better to consider the algebraic group $\mathrm{G}_n$ that is isomorphic, as an algebraic group defined over $\Q$, to $\mathrm{GO}(2, n)$, given by
\[
	\mathrm{G}_n(A) =  \left\{ g \in \mathrm{GL}_{n+2}(A) \mid g^t J_n g = \nu(g) J_n, \nu(g) \in A^\times \right\}, \mbox{ for every } \Q \mbox{-algebra } A\,,
\]
where 
\[
J_n = \left[ \begin{array}{ccccc}
 & & & & 1\\
 & & & 1 & \\
 & & Id_{n-2} & & \\
 & 1 & & & \\
 1 & & & & \end{array} \right].
\]
In fact, it can be verified that the conjugation, inside $\mathrm{GL}_{n+2}$, by the element
\[
D = \left[ \begin{array}{ccccc}
 1 & 0 & 0 & 0 & 1\\
 0 & 1 & 0& 1 & 0\\
 0 & 0 & Id_{n-2} & 0 & 0 \\ 
 0 & -1 & 0 & 1 & 0 \\
 -1 & 0 & 0 & 0 & 1 \end{array} \right],
\]
gives an isomorphism between the groups $\mathrm{GO}(2, n)$ and $\mathrm{G}_n$ defined over $\Q$. We introduced the group $\mathrm{GO}(2, n)$ because it allows to give an explicit and simple description of $h$. From now on we will work with the group $\mathrm{G}_n$ (and in this setting, the corresponding morphism $\mathbb{S}(\mathbb{R}) \rightarrow \mathrm{G}_n(\mathbb{R})$ is given by $z \mapsto Dh(z)D^{-1}$).

We denote by $\mathbb{A}_f$ the ring of finite adeles and by $K_\infty$ the centralizer  in $\mathrm{G}_n(\mathbb{R})$ of the morphism $DhD^{-1}:\mathbb{S}(\mathbb{R}) \rightarrow \mathrm{G}_n(\mathbb{R})$. Let $K_f \subset \mathrm{G}_n(\mathbb{A}_f)$ be an open compact subgroup, we denote $K = K_\infty \times K_f \subset \mathrm{G}_n(\mathbb{A})$ and define by $$S_K = \mathrm{G}_n(\mathbb{Q}) \backslash \mathrm{G}_n(\mathbb{R}) \times \mathrm{G}_n(\mathbb{A}_f) / K_\infty \times K_f$$ its corresponding level variety and by \[S = \lim_{{\overset{\longleftarrow}{K}}} S_K\] the space of complex points of the Shimura variety defined by this Shimura datum.

\section{Root system, $\Q$-Parabolic Subgroups and Irreducible Representations} \label{rsystem}

Consider the maximal $\Q$-split torus 
\[
\mathrm{A} = \left\{h \left[ \begin{array}{ccccc}
 h_1 & 0 & 0 & 0 & 0\\
 0 & h_2 & 0 & 0 & 0\\
 0 & 0 & Id_{n-2} & 0 & 0 \\ 
 0 & 0 & 0 & h_2^{-1} & 0 \\
 0 & 0 & 0 & 0 & h_1^{-1} \end{array} \right] : h, h_1, h_2 \in \mathbb{G}_m \right\} \subset \mathrm{G}_n
\] 
Let $\mathfrak{a}$ and $\mathfrak{g}_n$ denote the Lie algebra of $\mathrm{A}$ and $\mathrm{G}_n$, respectively. The corresponding $\mathbb{Q}$-root system $\Phi(\mathfrak{g}_n, \mathfrak{a})$ is of type $B_2$ and $\Delta_\mathbb{Q}=\left\{\varepsilon_1 - \varepsilon_2, \varepsilon_2\right\}$, where $\varepsilon_1, \varepsilon_2 \in \mathfrak{a}^\ast$ denote the usual elements, is a system of simple roots. This determines a set of proper standard $\mathbb{Q}$-parabolic subgroups $\mathcal{P}(G_{n})_\mathbb{Q}=\left\{\mathrm{P}_0, \mathrm{P}_1, \mathrm{P}_2 \right\}$, given by
\[ 
\mathrm{P}_1(\C) = \left\{ \left[ \begin{array}{ccccc}
\ast & \ast & \ldots & \ast & \ast \\
0 & \ast & \ldots & \ast & \ast \\
\vdots & \vdots & \ddots& \vdots & \vdots \\
0 & \ast & \ldots & \ast & \ast \\ 
0 & 0 & \ldots & 0 & \ast \end{array}  \right] \in \mathrm{GL}(n+2, \mathbb{C}) \right\} \cap \mathrm{G}_n(\mathbb{C}),
\]
\[ 
\mathrm{P}_2(\C) = \left\{ \left[ \begin{array}{ccccccc}
\ast & \ast & \ast & \ldots & \ast & \ast & \ast \\
\ast & \ast & \ast & \ldots & \ast & \ast & \ast \\
0 & 0 & \ast & \ldots & \ast & \ast & \ast \\
\vdots & \vdots & \vdots &\ddots & \vdots & \vdots & \vdots \\
0 & 0 & \ast & \ldots & \ast & \ast & \ast \\ 
0 & 0 & 0 & \ldots & 0 & \ast & \ast \\ 
0 & 0 & 0 & \ldots & 0 & \ast & \ast \end{array}  \right] \in \mathrm{GL}(n+2, \mathbb{C}) \right\} \cap \mathrm{G}_n(\mathbb{C})
\]
and $\mathrm{P}_0 = \mathrm{P}_1 \cap \mathrm{P}_2$. Let $\mathrm{A}_{\mathrm{P}_0}, \mathrm{A}_{\mathrm{P}_1}, \mathrm{A}_{\mathrm{P}_2} \subset \mathrm{A}$ be the following $\mathbb{Q}$-subtori:
\[ 
\mathrm{A}_{\mathrm{P}_1} = 
\left\{h \left[ \begin{array}{ccccc}
 h_1 & 0 & 0 \\
 0 & Id_{n} & 0 \\ 
 0 & 0 & h_1^{-1} \end{array} \right] : h, h_1 \in \mathbb{G}_m \right\},
\]
\[ 
\mathrm{A}_{\mathrm{P}_2} = \left\{h \left[ \begin{array}{ccccc}
 h_2Id_2 & 0 & 0 \\
 0 & Id_{n-2} & 0 \\ 
 0 & 0 & h_2^{-1}Id_2 \end{array} \right] : h, h_2 \in \mathbb{G}_m \right\},
\]
and $\mathrm{A}_{\mathrm{P}_0}= \mathrm{A}$. Finally, for $i \in \left\{0, 1, 2\right\}$, the Levi quotient $\mathrm{M}_i$ of $\mathrm{P}_i$ is canonically isomorphic to the centralizer $\mathrm{Z}_{\mathrm{G}_n}(\mathrm{A}_{\mathrm{P}_i})$ of $\mathrm{A}_{\mathrm{P}_i}$ in $\mathrm{G}_n$ (so we will use the same notation $\mathrm{M}_i$ for both groups).
One can see that over $\C$ the group $\mathrm{G}_n$ is isomorphic to the group of orthogonal similitudes $\mathrm{GO}(n+2)$ of matrices preserving the quadratic form defined by the matrix
\[
\left[ \begin{array}{ccccc}
 0 & 0 & 0 & 0 & 1\\
 0 & 0 & 0 & 1 & 0\\
 0 & 0 & \reflectbox{$\ddots$} & 0 & 0 \\ 
 0 & 1 & 0 & 0 & 0 \\
 1 & 0 & 0 & 0 & 0 \end{array} \right]
\]
(in dimension $n+2$) up to a scalar multiple. An isomorphism between $\mathrm{G}_n$ and $\mathrm{GO}(n+2)$ can be established by conjugation by a certain matrix of the form
\[ 
\left[ \begin{array}{ccccc}
 1 & 0 & 0 & 0 & 0\\
 0 & 1 & 0 & 0 & 0\\
 0 & 0 & M & 0 & 0 \\ 
 0 & 0 & 0 & 1 & 0 \\
 0 & 0 & 0 & 0 & 1 \end{array} \right]
\]
where $M \in \mathrm{GL}_{n-2}(\C)$ is given by
\[ 
\frac{1}{\sqrt{2}} \left[ \begin{array}{cccccc}
 1 & 0 & 0 & 0 & 0 & -i\\
 0 & \ddots & 0 & 0 & \reflectbox{$\ddots$} & 0\\
 0 & 0 & 1 & -i & 0 & 0\\ 
 0 & 0 & 1 & i & 0 & 0 \\ 
 0 & \reflectbox{$\ddots$} & 0 & 0 & \ddots & 0 \\
 1 & 0 & 0 & 0 & 0 & i \end{array} \right]  \quad \mbox{and}  \quad  \frac{1}{\sqrt{2}} \left[ \begin{array}{ccccccc}
 1 & 0 & 0 & 0 & 0 & 0 & -i\\
 0 & \ddots & 0 & 0 & 0 & \reflectbox{$\ddots$} & 0\\
 0 & 0 & 1 & 0 & -i & 0 & 0\\
 0 & 0 & 0 & 1 & 0 & 0 & 0\\  
 0 & 0 & 1 & 0 & i & 0 & 0 \\ 
 0 & \reflectbox{$\ddots$} & 0 & 0 & 0 & \ddots & 0 \\
 1 & 0 & 0 & 0 & 0 & 0 & i \end{array} \right],
\] 
if $n$ is even and odd respectively. The point is the following. We will study the cohomology spaces of the Shimura variety $S$ with respect to the local systems defined by absolutely irreducible representations of $\mathrm{G}_n$, that is by representations of $\mathrm{G}_n$ that are irreducible over $\mathbb{C}$. These are therefore the same as the absolutely irreducible representations of $\mathrm{GO}(n+2)$. On the other hand the classification of the irreducible representations of $\mathrm{GO}(n+2)$ is easier to obtain. There is a canonical maximal torus $\mathrm{T}$ in $\mathrm{GO}(n+2)$ which is given by the subgroup of all its diagonal matrices. It is clear that, under the aforementioned isomorphism $\mathrm{G}_n \cong \mathrm{GO}(n+2)$ the maximal $\Q$-split torus $\mathrm{A}$ is contained in $\mathrm{T}$ (this is important because of the compatibility condition between $\mathrm{A}$ and $\mathrm{T}$ enunciated in the introduction). Let $\mathfrak{t}$ be the Lie algebra associated to $\mathrm{T}$ then $\mathfrak{t}$ is given by all the diagonal elements in the Lie algebra $\mathfrak{g}=\mathfrak{go}(n+2)$ corresponding to $\mathrm{GO}(n+2)$. For the study of the corresponding root system and the irreducible representations we need to treat the cases $n$ odd and $n$ even separately. In what follows, $\mathfrak{t}_\mathbb{C}$ is the Lie algebra of $\mathrm{T}(\mathbb{C})$.

\subsection{Case $n$ odd} \label{rsystemodd}

$\mathrm{T}$ is a torus of dimension $l + 1$, where $l = \frac{n+1}{2}$. Now, we describe the irreducible finite dimensional representations of $\mathrm{G}_n$. One can see that $\mathfrak{go}(n+2)_\mathbb{C} = \mathfrak{so}(n+2)_\mathbb{C} \oplus \mathbb{C} Id_{n+2}$. On the other hand, let $\mathfrak{t}'_\mathbb{C} \subset \mathfrak{so}(n+2)_\mathbb{C}$ be the $l$-dimensional subspace of diagonal matrices. Here $\mathfrak{t}'_\mathbb{C}$ is a Cartan subalgebra of $\mathfrak{so}(n+2)_\mathbb{C}$. We consider the canonical coordinate elements $\epsilon'_1, \ldots, \epsilon'_l \in (\mathfrak{t}_\mathbb{C}')^\ast$. Then one knows that the corresponding root system is of type $B_l$ and $\Delta = \left\{ \epsilon'_1 - \epsilon'_2, \dots, \epsilon'_{l-1} - \epsilon'_l, \epsilon'_l \right\}$ is a system of simple roots. With respect to this choice of system of simple roots, the fundamental weights for $\mathfrak{so}(n+2)$ are given by:
\[
	\varpi_k = \sum_{i=1}^k \epsilon'_i, \quad \mbox{for } 1 \leq k < l \quad \mbox{ and } \quad \varpi_l = \frac{1}{2} \sum_{i=1}^l \epsilon'_i 
\]
and the finite dimensional irreducible representations of $\mathfrak{so}(n+2)$ are determined by their highest weights, given by the expressions of the form $n_1\varpi_1 + \ldots + n_l \varpi_l $ with $n_1, \ldots, n_l \in \mathbb{N}$. One says that such a representation is regular if $n_i > 0$ for all $i \in \left\{1, \ldots, l \right\}$. Only the highest weights with $n_l$ even will correspond to a finite dimensional irreducible representation of $\mathrm{SO}(n+2)$ (see for example, Proposition $3.1.19$ and Theorem $5.5.21$ of \cite{WalGood2009}). In other words, the irreducible finite dimensional representations of $\mathrm{SO}(n+2)$ can be determined by their highest weights and these are given by the elements of the form $a_1 \epsilon'_1 + \ldots + a_l \epsilon'_l$ with $a_1 \geq \ldots \geq a_l \in \mathbb{N}$. With respect to the decomposition $\mathfrak{t}_\mathbb{C} = \mathfrak{t}'_\mathbb{C} \oplus \mathbb{C} Id_{n+2}$, let $\epsilon_i \in \mathfrak{t}_\mathbb{C}^\ast$, for each $i \in \left\{1, \ldots, l\right\}$, be the extension of $\epsilon'_i$ by zero on the second component and let $\kappa \in \mathfrak{t}_\mathbb{C}^\ast$ be the element that is zero in the first component  and such that $\kappa(zId_{n+2}) = z$. From the fact that $\mathrm{GO}(n+2)$ is the direct product of its center $\mathrm{Z}$ ($\cong \mathbb{G}_m$) and $\mathrm{SO(n+2)}$, one can deduce that the finite dimensional irreducible representations of $\mathrm{GO}(n+2)$ are in bijection with the highest weights of the form
 $a_1 \epsilon_1 + \ldots + a_l \epsilon_l + c\kappa$ with $a_1 \geq \ldots \geq a_l \in \mathbb{N}$ and $c \in \mathbb{Z}$.

Finally, with respect to the root system defined by $\mathfrak{t}$, the Weyl group $\mathcal{W}=\mathcal{W}(\mathfrak{go}(n+2)_\mathbb{C}, \mathfrak{t}_\C)$ has $2^l l!$ elements and these elements are given by the composition of a permutation in $S_l$ acting on $\left\{ \epsilon_1, \ldots, \epsilon_l \right\}$ and any possible change of signs on these elements. For a given permutation $\sigma \in S_l$ and $f:\left\{ 1, \ldots, l\right\} \rightarrow \left\{ 1, -1 \right\}$, we denote by $w = w_{\sigma, f}$ the element in $\mathcal{W}$ that takes each $\epsilon_i$ to $f(\sigma(i)) \epsilon_{\sigma(i)}$.

\subsection{Case $n$ even} \label{rsystemeven}

Following a similar procedure, we can determine the irreducible finite dimensional representations of $\mathrm{G}_n$ by their corresponding highest weights. In this case $l = \frac{n+2}{2}$ and $\mathrm{T}$ has dimension $l + 1$. Let $\mathfrak{t}'_\mathbb{C} \subset \mathfrak{so}(n+2)_\mathbb{C}$ be, again, the $l$-dimensional subspace of diagonal matrices, then $\mathfrak{t}'_\mathbb{C}$ is a Cartan subalgebra of $\mathfrak{so}(n+2)_\mathbb{C}$. The corresponding root system is of type $D_l$ and $\Delta = \left\{ \epsilon'_1 - \epsilon'_2, \dots, \epsilon'_{l-1} - \epsilon'_l, \epsilon'_{l-1} + \epsilon'_l \right\}$ is a system of simple roots, where $\epsilon'_1, \ldots, \epsilon'_l \in (\mathfrak{t}_\mathbb{C}')^\ast$ is the canonical base in $(\mathfrak{t}_\mathbb{C}')^\ast$. Therefore the fundamental weights for $\mathfrak{so}(n+2)$ are given by:
\[
	\varpi_k = \sum_{i = 1}^k \epsilon'_i, \quad \mbox{for } 1 \leq k < l-1, \quad \varpi_{l-1} = \frac{1}{2}\left(\sum_{i = 1}^{l-1} \epsilon'_i - \epsilon'_l \right) \mbox{ and } \varpi_l = \frac{1}{2} \sum_{i = 1}^{l} \epsilon'_i
\]
and the finite dimensional irreducible representations of $\mathfrak{so}(n+2)$ are determined by their highest weights, given by the expressions of the form $n_1\varpi_1 + \ldots + n_l \varpi_l $ with $n_1, \ldots, n_l \in \mathbb{N}$. One says that such a representation is regular if $n_i > 0$ for all $i \in \left\{1, \ldots, l \right\}$. Among these highest weights, only those with $n_{l-1}+n_l$ even will correspond to a finite dimensional irreducible representation of $\mathrm{SO}(n+2)$.

In other words, the finite dimensional irreducible representations of $\mathrm{SO}(n+2)$ are determined by their highest weights, that are of the form
$a_1 \epsilon'_1 + \ldots + a_l \epsilon'_l$ where $a_1 \geq \ldots \geq a_{l-1} \geq \left|a_l\right| \in \mathbb{N}$.

In this case $\mathrm{GO}(n+2)$ is the semidirect product of its center $\mathrm{Z}$ and $\mathrm{SO}(n+2)$, and their intersection is $\left\{ \pm Id_{n+2}\right\}$. We define the elements $\epsilon_1, \ldots, \epsilon_l, \kappa \in \mathfrak{t}_\mathbb{C}^\ast$ as in Subsection \ref{rsystemodd}. One can finally deduce that the finite dimensional irreducible representations of $\mathrm{G}_n$ are in bijection with the highest weights of the form $a_1 \epsilon_1 + \ldots + a_l  \epsilon_l + c\kappa$ with $a_1 \geq \ldots \geq a_{l-1} \geq \left| a_l \right| \in \mathbb{N}$ and $c \in \mathbb{Z}$ with $c \equiv a_1+a_2+ \ldots + a_{l}$ (mod $2$), where the congruence modulo $2$ is the compatibility condition between the representation of $\mathrm{SO}(n+2)$ and the character on the center.

The Weyl group $\mathcal{W}$ has $2^{(l-1)} l!$ elements. It is given by all compositions of an element of the group of permutations $S_l$ on $\left\{ \epsilon_1, \ldots, \epsilon_l \right\}$ and a change of sign on an even number of these elements. For a given permutation $\sigma \in S_l$ and $f:\left\{ 1, \ldots, l\right\} \rightarrow \left\{ 1, -1 \right\}$, we use the same notation as in the last subsection to denote the corresponding element $w = w_{\sigma, f}$ in the Weyl group.

\section{Weyl Representatives}\label{wr}

In this section we describe the set of Weyl representatives associated to each standard $\mathbb{Q}$-parabolic subgroup of $\mathrm{G}_n$ as defined in \cite{Kostant61}. $\Delta_i$ will denote the set of roots appearing in the Lie algebra of the unipotent radical of the parabolic subgroup $\mathrm{P}_i$ of $\mathrm{G}_n$, for $i \in \left\{0, 1, 2 \right\}$. Because of the difference between the corresponding Weyl groups, the even and odd cases will be treated separately.

\subsection{Case $n$ odd} \label{weylodd}

We begin with the description of the Weyl representatives for the minimal $\Q$-parabolic subgroup $\mathrm{P}_0$. The roots appearing in the unipotent radical of $\mathrm{P}_0$ are 
\[
\Delta_{0} = \left\{ \epsilon_1 \pm \epsilon_2, \ldots, \epsilon_1 \pm \epsilon_l,  \epsilon_2 \pm \epsilon_3, \ldots, \epsilon_2 \pm \epsilon_l, \epsilon_1, \epsilon_2 \right\} 
\]
and by definition the set of Weyl representatives $\mathcal{W}^{\mathrm{P}_0}$ are the elements $w \in \mathcal{W}$ such that $w(\Phi^-) \cap \Phi^+ \subset \Delta_{0}$, but the elements in $\Phi^{+}$ which are not in $\Delta_{0}$ are $
\Phi^+ \setminus \Delta_{0} = \left\{ \epsilon_m \pm \epsilon_n, \epsilon_m \mid 2 < m < n \leq l \right\}.$ From this fact one can see the following:
\begin{lema}
Let $w_{\sigma, f}$ be an element of the Weyl group $\mathcal{W}$, then $w_{\sigma, f} \in \mathcal{W}^{\mathrm{P}_0}$ if and only if
\begin{enumerate}
 \item  $f(m) = 1 \quad \forall m > 2$, and
 \item $\sigma^{-1}(m) < \sigma^{-1}(n) \quad \mbox{for } 2 < m < n \leq l$.
\end{enumerate} 
\end{lema}

In fact $w_{\sigma, f} \in \mathcal{W}^{\mathrm{P}_0}$ is determined by the values $f(1), f(2), \sigma^{-1}(1)$ and $\sigma^{-1}(2)$. Therefore $ \mathcal{W}^{\mathrm{P}_0}$ has $4(l-1) l$ elements. Observe that the only element in $\Delta_{0}$ which is not in $\Delta_{2}$ is $\epsilon_1 - \epsilon_2$. Then, clearly $\mathcal{W}^{\mathrm{P}_2}$ is the subset of $\mathcal{W}^{\mathrm{P}_0}$ of Weyl elements $w$ such that $\epsilon_1 - \epsilon_2 \notin w(\Phi^-)$. From this fact one can easily see that, for $w_{\sigma, f} \in \mathcal{W}^{\mathrm{P}_0}$, if $f(1) = -1$ and $f(2) = 1$ then $w_{\sigma, f} \notin \mathcal{W}^{\mathrm{P}_2}$. On the other hand, if $f(1) = 1$ and $f(2) = -1$ then for any $\sigma \in S_l$, the corresponding element $w_{\sigma, f} \in \mathcal{W}^{\mathrm{P}_0}$. Moreover we see the following

\begin{lema}
$\mathcal{W}^{\mathrm{P}_2}$ consists of the elements $w_{\sigma, f} \in \mathcal{W}^{\mathrm{P}_0}$ satisfying one of the following conditions
\begin{itemize}
\item $f(1) = 1$ and $f(2) = -1$.
\item $f(1) = f(2) = 1$ and $\sigma^{-1}(1) < \sigma^{-1}(2)$.
\item $f(1) = f(2) = -1$ and $\sigma^{-1}(1) > \sigma^{-1}(2)$.
\end{itemize}
\end{lema}

Finally, $\Delta_{1} = \left\{ e_1 \pm e_2, \ldots, e_1 \pm e_l, e_1 \right\}$ and using the above methods, we get the following

\begin{lema}
$\mathcal{W}^{\mathrm{P}_1}$ consists of the elements $w_{\sigma, f} \in \mathcal{W}^{\mathrm{P}_0}$ satisfying the following conditions
\begin{itemize}
\item $f(2) = 1$ and
\item $\sigma^{-1}(2) < \sigma^{-1}(3)$.
\end{itemize}
\end{lema}

In particular, if $l = 3$, $w_{\sigma, f} \in \mathcal{W}^{\mathrm{P}_1}$ if $f(2) = f(3) = 1$ and $\sigma \in \left\{id, (12), (123)\right\}$.
One can observe the similarity with the description of the Weyl representatives in Proposition $8$ of \cite{GotGro2009}.

We now describe the Weyl representatives $\mathcal{W}_{i}^{0}$ of $\mathrm{P}_0 \cap \mathrm{M}_i$ in $\mathrm{M}_i$ for $i=1,2$. Using the same methods as above, we determine $\mathcal{W}_2^0 = \left\{w_{e, \mathfrak{1}}, w_{(1, 2), \mathfrak{1}}\right\}$, where $\mathfrak{1}$ denotes here the constant function that takes always the value $1$, and 
\[
\mathcal{W}_1^0 = \left\{w_{\sigma, f} \mid f(m)=1 \quad \forall m \neq 2, \sigma(1)=1 \mbox{ and } \sigma^{-1}(m) < \sigma^{-1}(n) \quad \forall 2 < m < n \leq l \right\}\,.
\]

Note that $\mathcal{W}^{\mathrm{P}_0}=\mathcal{W}_0^i\mathcal{W}^{\mathrm{P}_i}$ for $i=1, 2$.

\subsection{Case $n$ even} \label{weyleven}

In this case, $\Delta_0$ is given by $\left\{ \epsilon_1 \pm \epsilon_k \mid 1 < k \leq l \right\} \cup \left\{ \epsilon_2 \pm \epsilon_k \mid 2 < k \leq l \right\}$. For $2 < k < l$, we see that if $f(k)=-1$ then $w_{\sigma, f} \notin \mathcal{W}^{\mathrm{P}_0}$ (because this element takes the root $-\epsilon_{\sigma^{-1}(k)} - \epsilon_{\sigma^{-1}(l)}$ to a positive root not in $\Delta_0$). In fact we get the following 
\begin{lema} 
$\mathcal{W}^{\mathrm{P}_0}$ is given by all the elements $w_{\sigma, f} \in \mathcal{W}$ satisfying
\begin{enumerate}
\item $f(k) = 1 \quad \mbox{for } 2 < k < l$.
\item $\sigma^{-1}(m) < \sigma^{-1}(n) \quad \mbox{for } 2 < m < n \leq l$.
\end{enumerate} 
\end{lema}

Also, the only element in $\Delta_0$ which is not in $\Delta_2$ is $\epsilon_1 - \epsilon_2$ and $\Delta_1= \left\{ \epsilon_1 \pm \epsilon_k \mid 1 < k \leq l \right\}$. From these facts we deduce the following

\begin{lema} 
$\mathcal{W}^{\mathrm{P}_2}$ is the subset of $\mathcal{W}^{\mathrm{P}_0}$ consisting of the elements $w_{\sigma, f} \in \mathcal{W}^{\mathrm{P}_0}$ satisfying one of the following conditions :
\begin{enumerate}
\item $f(1)=1$ and $f(2)=-1$.
\item $f(1)=f(2)=1$ and $\sigma^{-1}(1) < \sigma^{-1}(2)$.
\item $f(1)=f(2)=-1$ and $\sigma^{-1}(1) > \sigma^{-1}(2)$.
\end{enumerate}
and $\mathcal{W}^{\mathrm{P}_1}$ is the subset of $\mathcal{W}^{\mathrm{P}_0}$ consisting of the elements $w_{\sigma, f} \in \mathcal{W}^{\mathrm{P}_0}$ satisfying both conditions
\begin{enumerate}
\item $f(2)=1$.
\item $\sigma^{-1}(2) < \sigma^{-1}(3)$.
\end{enumerate}
\end{lema}

By similar computations, the sets $\mathcal{W}_2^0$ and $\mathcal{W}_1^0$ are given by $\mathcal{W}_2^0 = \left\{ w_{e, \mathfrak{1}}, w_{(1, 2), \mathfrak{1}}\right\}$ and $$\mathcal{W}_1^0 = \left\{ w_{\sigma, f} \in \mathcal{W}^{\mathrm{P}_0} \mid \sigma(1)=1, f(1)=1\right\}\,.$$

\section{Mixed Hodge Theory}\label{mht}

We now collect some information regarding the weight filtration of the mixed Hodge structure for the cohomology spaces in the long exact sequences~\eqr{first} and~\eqr{second}.

First of all, the weight morphism of the orthogonal Shimura variety associated to $\mathrm{GO}(2, n)$ is given by the morphism $\omega :\mathbb{G}_m \rightarrow \mathrm{GO}(2, n)$ defined by $t \mapsto t^2Id_{n+2}$. Therefore, for a finite dimensional irreducible representation $(\rho_{_\lambda}, V_\lambda)$ with highest weight $\lambda = \sum_{i=1}^l a_i \epsilon_i + c \kappa$, the composition $\rho_{_\lambda} \circ \omega : \mathbb{G}_m \rightarrow \mathrm{GL}(V_\lambda)$ is given by $t \mapsto t^{2c}Id_{n+2}$. Therefore $V_\lambda$ defines a complex variation of Hodge structure of weight $-2c$ and the mixed Hodge structure on the space
$H^q(S, \V_\lambda)$ has weights greater than or equal to $q-2c$ (see Theorem $2.2.7$ of \cite{Harris94}).

We continue by calculating, for each $i \in \left\{0, 1, 2\right\}$, the morphism $h_i:\mathbb{S} \rightarrow \mathrm{G}_{h, i}$ defining a Shimura pair $(\mathrm{G}_{h, i}, h_i)$ where $\mathrm{G}_{h, i}$ is the Hermitian part of the Levi subgroup $\mathrm{M}_i$ of $\mathrm{P}_i$. For this we use the description given in \cite{Harris86} (but one could also use Chapter 4 of~\cite{Pink}). 

First of all, we need to introduce some notation. Given an algebraic representation $\rho:\mathrm{G}_n \rightarrow \mathrm{GL}(V)$ defined over $\mathbb{Q}$ one has:
\begin{itemize}
\item A decreasing filtration $F^\bullet_h V_\mathbb{C}$ of $V_\mathbb{C} = V \otimes_\mathbb{Q} \mathbb{C}$ defined by the composition $\rho \circ h : \mathbb{S} \rightarrow GL(V)$ by 
\[
    F_h^p V_\mathbb{C} = \oplus_{p' \geq p} V^{p', q}
\]
where for every $p, q \in \mathbb{Z}$, $V^{p, q} = \left\{v \in V_\mathbb{C} \mid \rho \circ h(z) v = z^{-p}\bar{z}^{-q} v \right\}$ (where $h, z$ and $\bar{z}$ are as in Section \ref{tsv}). 
\item Every morphism $\chi:\mathbb{G}_m \rightarrow \mathrm{G}_n$ defined over $\mathbb{Q}$, defines an increasing filtration $W_\bullet^\chi V$ given by
\[
    W_n^\chi V = \oplus_{n' \leq n} V_{n'}^\chi
\]
where for every $n \in \mathbb{Z}$, $V_n^\chi = \left\{v \in V \mid \rho \circ \chi(r) v = r^{n}v \right\}$.
\end{itemize}

Let $\omega^{\mathrm{P}_i}$ be the unique admissible Cayley morphism $\omega^{\mathrm{P}_i}: \mathbb{G}_m \rightarrow \mathrm{A}_{\mathrm{P}_i}$ (see Theorem $5.1.3$ of \cite{Harris86}). In particular, this morphims satisfies:
\begin{itemize}
\item For every representation $\rho : \mathrm{G}_n \rightarrow \mathrm{GL}(V)$ defined over $\mathbb{Q}$, the pair of filtrations $(W_\bullet^{\omega^{\mathrm{P}_i}} V, F^\bullet_h V)$ defines a mixed Hodge structure on $V$.
\item Let $\mathrm{U}_i \subset \mathrm{P}_i$ be the unipotent radical and $\mathrm{W}_i \subset \mathrm{U}_i$ be the center of $\mathrm{U}_i$. For the adjoint representation of $\mathrm{G}_n$ on its Lie algebra $\mathfrak{g}_n$, the filtration $W_\bullet^{\omega^{\mathrm{P}_i}} \mathfrak{g}_n$ satisfies that
$(\mathfrak{g}_n)_{-2}^{\omega^{\mathrm{P}_i}}$ is the Lie algebra of $\mathrm{W}_i$, $(\mathfrak{g}_n)_{-1}^{\omega^{\mathrm{P}_i}} \oplus (\mathfrak{g}_n)_{-2}^{\omega^{\mathrm{P}_i}}$ is the Lie algebra of $\mathrm{U}_i$  and $(\mathfrak{g}_n)_0^{\omega^{\mathrm{P}_i}}$ is the Lie algebra of $\mathrm{M}_i$.
\end{itemize}

From now on, for every representation $\rho:\mathrm{G}_n \rightarrow \mathrm{GL}(V)$, we will denote by $W_\bullet^{\mathrm{P}_i} V$ and $V_\bullet^{\mathrm{P}_i}$ the filtration and the graduation on $V$, respectively, defined by the admissible Cayley morphism $\omega^{\mathrm{P}_i}$. 

Once we determine the admissible Cayley morphism by using the aforementioned properties, it will be enough in our case to use 5.1.9 of \cite{Harris86} with the standard representation, given by the natural inclusion $\mathrm{G}_n \hookrightarrow \mathrm{GL}_{n+2}$, to calculate $h_i$.

\subsection{Case $i=1$}

Let $\omega^{\mathrm{P}_1}:\mathbb{G}_m \rightarrow \mathrm{A}_{\mathrm{P}_1}$ be the unique admissible Cayley morphism. Because of the description of $\mathrm{A}_{\mathrm{P}_1}$, there exists $m, k \in \mathbb{Z}$ such that
\[
 \omega^{\mathrm{P}_1}(r) = \left[ \begin{array}{ccccc}
 r^{k+m} & 0 & 0 \\
 0  & r^k Id_n & 0 \\ 
 0 & 0 & r^{k-m} \end{array} \right] \quad \forall r \in \mathbb{C}^\times.
\]
By using the description of the filtration on the Lie algebra $\mathfrak{g}_n$ defined by the  composition of the adjoint representation with $\omega^{\mathrm{P}_1}$ and the fact that in this case, the unipotent radical is commutative (therefore $\mathrm{U}_1 = \mathrm{W}_1$), one finally has $m=-2$.

Now, consider the standard representation given by the inclusion $\mathrm{G}_n \hookrightarrow \mathrm{GL}_{2+n}$ and let $V=\mathbb{Q}^{n+2}$. We have defined $h:\mathbb{S} \rightarrow \mathrm{GO}(2, n)$, so we have to compose this morphism with the conjugation by $D$ in order to work with the group $\mathrm{G}_n$ ($D$ as in Section \ref{tsv}) and consider the filtration defined by $DhD^{-1}$ on $V \otimes_\mathbb{Q} \mathbb{C}$. In particular, to get this filtration one can apply $D$ to the Hodge filtration defined by the morphism $h:\mathbb{S} \rightarrow \mathrm{GO}(2, n) \hookrightarrow \mathrm{GL}_{2+n}$ on $\mathbb{C}^{n+2}$. Then the Hodge filtration $F^\bullet_h V$ on $V_\mathbb{C}$ is defined by the graduation
\[
V^{p, q} = \left\{ \begin{array}{ccccc}
 \langle De_1 - iDe_2 \rangle = \langle e_1 - e_{n+2} - ie_2 + ie_{n+1} \rangle,, \quad \mbox{ if } (p, q) =(0, -2) \\
 \langle De_3, \ldots, De_n, De_{n+1}, De_{n+2} \rangle =\langle e_3, \ldots, e_n, e_{n+1}+e_2, e_{n+2}+e_1 \rangle, \quad \mbox{ if } (p, q) =(-1, -1)  \\
 \langle De_1 + iDe_2 \rangle = \langle e_1 - e_{n+2} + ie_2 - ie_{n+1} \rangle, \quad \mbox{ if } (p, q) =(-2, 0)  
 \end{array} \right.
\]
and the weight filtration $W_\bullet^{\mathrm{P}_1} V$ on $V$ is defined by the graduation
\[
V_j^{\mathrm{P}_1} = \left\{ \begin{array}{ccccc}
 \langle e_{n+2} \rangle , \quad \mbox{ if } j = k+2  \\
 \langle e_2, \ldots, e_{n+1} \rangle, \quad \mbox{ if } j = k  \\
 \langle e_{1} \rangle , \quad \mbox{ if } j=k-2 
 \end{array} \right.
\]

One can see that the Hodge filtration $F^\bullet_h V$ induces on $W_{k-2}V = V_{k-2}^{\mathrm{P}_1}$ the filtration 
\[
F^{j} V^{\mathrm{P}_1}_{k-2} = F^j_h\mathbb{C}^{n+2} \cap (V_{k-2}^{\mathrm{P}_1} \otimes_\mathbb{Q} \mathbb{C}) = \left\{ \begin{array}{ccccc}
 V_{k-2}^{\mathrm{P}_1} \otimes_\mathbb{Q} \mathbb{C} , \quad \mbox{ if } j = -2  \\
 0 , \quad \mbox{ if } j=-1 
 \end{array} \right.
\]
On the other hand, the Hodge filtration must define a Hodge structure of weight $k-2$ on $W_{k-2}V$. This implies that $k = -2$.

Now, by using $5.1.9$ of \cite{Harris86} one finally can see, by using the standard representation, that the morphism $h_1:\mathbb{S} \rightarrow \mathrm{G}_{h, 1} \subset \mathrm{GO}(2, n)$ is given by
\[
h_1(z) = \left[ \begin{array}{ccc}
\left| z \right|^4 & & \ \\
 & \left| z \right|^2 Id_n &  \\
 &  &  1 \end{array}  \right] \,, \quad  \forall z \in \mathbb{S}(\mathbb{R}). 
\] 
In particular, the weight morphism associated to $(\mathrm{G}_{h, 1}, h_1)$ is the morphism $\omega_1:\mathbb{G}_m \rightarrow \mathrm{G}_{h, 1} \subset \mathrm{M}_1$ given by
\[
\omega_1(t) = \left[ \begin{array}{ccc}
t^4 & & \ \\
 & t^2 Id_n &  \\
 &  &  1 \end{array}  \right] = t^2 
 \left[ \begin{array}{ccc}
t^2 & & \ \\
 & Id_n &  \\
 &  &  t^{-2} \end{array}  \right] \,, \quad  \forall t \in \mathbb{G}_m(\mathbb{R}). 
\] 
From this description of the weight morphism $\omega_1$ one can see the following. Let $w \in \mathcal{W}^{\mathrm{P}_1}$ and let $w_\ast(\lambda) = n_1 \epsilon_1 + \ldots + n_l \epsilon_l + c \kappa$ be defined as in Subsection \ref{Decomposition}. if  $W_{w_\ast(\lambda)}$ is the irreducible representation of $\mathrm{M}_i$ with highest weight $w_\ast(\lambda)$, then the mixed Hodge structure on the space $H^q(S^{\mathrm{M}_1}, \widetilde{W}_{w_\ast(\lambda)})$, described in \cite{HaZu-II-94}, has weights greater than or equal to $q -2c - 2n_1$. 

\subsection{Case $i=2$}

In this case, by using the same procedure as in the case $i=1$, one has that $\mathrm{U}_2 \neq \mathrm{W}_2$ and by using the filtration that the Cayley morphism $\omega^{\mathrm{P}_2}$ induces on the Lie algebra of $\mathrm{G}_n$ one can see that 
\[
	\omega^{\mathrm{P}_2}(r) = \left[ \begin{array}{ccc}
r^{k-1} Id_2 & & \\ 
& r^k Id_{n-2} & \\
& & r^{k+1} Id_2  \end{array}  \right] \quad  \forall r \in \mathbb{S}(\mathbb{R}), 
\]
for $k \in \mathbb{Z}$. Now consider the representation of $\mathrm{G}_n$ on $V=\mathbb{Q}^{n+2}$ given by the natural inclusion $\mathrm{G}_n \hookrightarrow \mathrm{GL}_{n+2}$. The property that the pair of filtrations $(W_\bullet^{\omega^{\mathrm{P}_2}} V, F^\bullet_h V)$ defines a mixed Hodge structure on $V$ implies that one has $k=-2$ and finally that $h_2:\mathbb{S} \rightarrow \mathrm{G}_{h, 2} \subset \mathrm{GO}(2, n)$ is given by
\[
	h_{2}(z) = \left[ \begin{array}{ccc}
\left| z \right|^2 \left[ \begin{array}{cc} 
x & y   \\
-y & x \end{array} \right] &  \\
 & \left| z \right|^2 Id_{n-2}   &  \\
 &  &  \left[ \begin{array}{cc} 
x & y   \\
-y & x \end{array} \right] \end{array}  \right] \quad  \forall z=(x+iy) \in \mathbb{S}(\mathbb{R}). 
\] 
Thus, the corresponding weight morphism is given by
\[
	\omega_2(t) = \left[ \begin{array}{ccc}
t^3 Id_2 & & \\
 & t^2 Id_{n-2} & \\
 & & t Id_2 \end{array} \right]
 = t^2 \left[ \begin{array}{ccc}
t Id_2 & & \\
 & Id_{n-2} &  \\
 &  &  t^{-1} Id_2 \end{array}  \right] \quad  \forall t \in \mathbb{G}_m(\mathbb{R}). 
\] 
One can deduce the following. For $w \in \mathcal{W}^{P_2}$ and $w_\ast(\lambda)=n_1 \epsilon_1 + \ldots + n_l \epsilon_l + c \kappa$ defined as in Subsection \ref{Decomposition}, the weights in the mixed Hodge structure associated to $H^q(S^{\mathrm{M}_2}, \widetilde{W}_{w_\ast(\lambda)})$ are greater than or equal to $q -2c - n_1 - n_2$. 

\subsection{Case $i=0$} \label{weight0}
In this case, one has that the parabolic subgroup $\mathrm{P}_0$ is subordinate (in the sence of section $2.2$ of \cite{HaZu-II-94}) to $\mathrm{P}_1$. Then the hermitian part of $\mathrm{P}_0$ is exaclty the hermitian part of $\mathrm{P}_1$ and, for $w \in \mathcal{W}^{\mathrm{P}_0}$ with $w_\ast(\lambda)=n_1 \epsilon_1 + \ldots + n_l \epsilon_l + c \kappa$, the mixed Hodge structure on the space $H^0(S^{\mathrm{M}_0}, \widetilde{W}_{w_\ast(\lambda)})$  has weight equal to $-2c - 2n_1$ (note that $S^{\mathrm{M}_0}$ can only have cohomology in degree zero).

\section{Important Facts}\label{facts}

For notational convenience, we use $\partial_{i}$ in place of $\partial_{\mathrm{P}_{i}}$ for $i\in\{0,1,2\}$. In this section, we explain the methods used to determine when a cohomology class in $H^\ast(\partial_0, \widetilde{V}_\lambda)$ does not contribute to a ghost class in the cohomology of the boundary. From now on, whenever $n$ is clear from the context, we will denote $\mathrm{G}_n$ simply by $\mathrm{G}$.

\subsection{A decomposition of $H^\bullet(\partial_i, \V_\lambda)$}\label{Decomposition}

In this subsection, a well known decomposition of the spaces $H^\bullet(\partial_i, \V_\lambda)$ is introduced. For $w \in \mathcal{W}$, we denote by $\ell(w)$ the length of $w$. For each $i \in \left\{0, 1, 2\right\}$ and $w \in \mathcal{W}^{\mathrm{P}_i}$, we write $w_\ast(\lambda) = w(\lambda + \rho) - \rho \in \mathfrak{h}^\ast$ where $\rho = \frac{1}{2}\sum_{\alpha \in \Phi^+} \alpha$. Then $w_\ast(\lambda)$ is the highest weight associated to an irreducible finite dimensional representation $W_{w_\ast(\lambda)}$ of $\mathrm{M}_i$.  For each $q \in \mathbb{N}$ we have,
\begin{equation} \label{eq:decomp}
H^{q}(\partial_i, \V_\lambda) = \bigoplus_{w \in \mathcal{W}^{\mathrm{P}_i}} Ind_{\mathrm{P}_i(\mathbb{A}_f)}^{\mathrm{G}(\mathbb{A}_f)} H^{q-\ell(w)}(S^{\mathrm{M}_i}, \widetilde{W}_{w_\ast(\lambda)}).
\end{equation}
where $Ind_{\mathrm{P}_i(\mathbb{A}_f)}^{\mathrm{G}(\mathbb{A}_f)}$ denotes the algebraic (unnormalized) induction and $S^{\mathrm{M}_i}$ is the symmetric space associated to $\mathrm{M}_i$. For the rest of this paper we will denote $Ind_{\mathrm{P}_i(\mathbb{A}_f)}^{\mathrm{G}(\mathbb{A}_f)}$ by $Ind_{\mathrm{P}_i}^\mathrm{G}$.

For each $q \in \mathbb{N}$,  let $\mathcal{W}^{\mathrm{P}_i}(q)$ be the set of the elements $w \in \mathcal{W}^{\mathrm{P}_i}$ with $\ell(w)=q$. Since $S^{\mathrm{M}_0}$ can only have nontrivial cohomology in degree $0$,
\begin{equation} \label{eq:Decomp}
H^{q}(\partial_0, \V_\lambda) = \bigoplus_{w \in \mathcal{W}^{\mathrm{P}_0}(q)} Ind_{\mathrm{P}_0}^{\mathrm{G}} H^{0}(S^{\mathrm{M}_0}, \widetilde{W}_{w_\ast(\lambda)}) \,, \quad \forall q \in \mathbb{N}.
\end{equation} 

In order to study $ker(p^q)$ (see~\eqr{second}), we study the image of the map $\delta_q:H^{q-1}(\partial_0, \V_\lambda) \rightarrow H^q(\partial \overline{S}, \V_\lambda)$. Therefore, for each $w \in \mathcal{W}^{P_0}(q-1)$ we study whether the space $Ind_{\mathrm{P}_0}^{\mathrm{G}} H^{0}(S^{\mathrm{M}_0}, \widetilde{W}_{w_\ast(\lambda)})$ is in the kernel of $\delta_q$ and, when this is not the case, whether it could contribute to ghost classes. 

\subsection{Middle weight} \label{MW}
The fact that the weights in the mixed Hodge structure on $H^q(S, \widetilde{V}_\lambda)$ are greater than or equal to $q - 2c$ is strongly used. Note that $-2c$ is the unique weight in the variation of complex Hodge structure defined by $V_\lambda$. If $w \in \mathcal{W}^{\mathrm{P_0}}$ and $w_\ast(\lambda) = n_1\epsilon_1 + n_2\epsilon_2+n_3\epsilon_3+c\kappa$, then the subspace $Ind_{\mathrm{P}_0}^{\mathrm{G}} H^0(S^{\mathrm{M}_0}, \widetilde{W}_{w_\ast(\lambda)})$ of $H^{q-1}(\partial_0, \V_\lambda)$ in ~\eqr{Decomp} has weight $-2n_1-2c$. Note that $\ell(w)=q-1$. Thus, a necesary condition for the space $Ind_{\mathrm{P}_0}^{\mathrm{G}} H^0(S^{\mathrm{M}_0}, \widetilde{W}_{w_\ast(\lambda)})$ to contribute to ghost classes is that $-2c - 2n_1 \geq q-2c  = \ell(w) + 1 - 2c$

We summarize the above discussion in the form of following lemma.

\begin{lema}\label{mwelimination} 
If $w \in \mathcal{W}^{\mathrm{P}_0}$ satisfies the inequality 
\begin{eqnarray}\label{eq:length}
 \ell(w) + 1 > -2n_1 \nonumber
\end{eqnarray}
then the space $Ind_{\mathrm{P}_0}^{\mathrm{G}} H^0(S^{\mathrm{M}_0}, \widetilde{W}_{w_\ast(\lambda)})$ cannot contribute to ghost classes in $H^{\bullet}(\partial \overline{S}, \V_\lambda)$.
\end{lema}

\subsection{Image of $r_i:H^\bullet(\partial_i, \V_\lambda) \rightarrow H^\bullet(\partial_0, \V_\lambda)$}  \label{r1}
  
To study the image of $r_i$ we use the general description of Eisenstein cohomology in ~\cite{Sch83} and ~\cite{Schwermer1994}. In order to enunciate the main theorem that will be used for the study of the images of the morphisms $r_i$, we need to introduce some notations.

From now on, $i$ will denote an element in $\left\{0, 1, 2\right\}$. As usual, we denote by $\mathrm{U}_i$ the unipotent radical of $\mathrm{P}_i$. We denote $d_i = dim (\mathrm{U}_0(\mathbb{R})/\mathrm{U}_i(\mathbb{R}))$. 

We use the following notations
\[
	\mathfrak{a}_{\mathrm{P}_i} = X_\ast(\mathrm{A}_\mathrm{P_i}) \otimes \mathbb{R}, \quad \check{\mathfrak{a}}_{\mathrm{P_i}} = X^\ast(\mathrm{P}_i) \otimes \mathbb{R}
\]
where $X_\ast(\mathrm{A}_\mathrm{P_i})$ and $X^\ast(\mathrm{P}_i)$ denote, respectively, the group of $\mathbb{Q}$-rational cocharacters of $\mathrm{A}_\mathrm{P_i}$ and the group of $\mathbb{Q}$-rational characters of $\mathrm{P_i}$. There is a natural isomorphism between $\mathfrak{a}_{\mathrm{P}_i}$ and the Lie algebra of $\mathrm{A}_\mathrm{P_i}(\mathbb{R})$, and $\check{\mathfrak{a}}_{\mathrm{P_i}}$ is naturally isomorphic to $\mathfrak{a}_{\mathrm{P}_i}^\ast$. The natural pairing between $\check{\mathfrak{a}}_{\mathrm{P}_i}$ and $\mathfrak{a}_{\mathrm{P}_i}$ will be denoted by $\langle, \rangle$. In particular, $\mathfrak{a}_{\mathrm{P}_0}$ is naturally isomorphic to $Lie(\mathrm{A}(\mathbb{R}))$. Remember, from Section ~\ref{tsv}, that $\varepsilon_1, \varepsilon_2$ denote the usual first and second coordinate functions in the diagonal matrices of $\mathrm{A}$. 

Let $\Delta^{\mathrm{P}_i}_{\mathrm{P}_0} \subset \Delta_\mathbb{Q}$ be the set of simple roots which occur in the Lie algabra of $\mathrm{U}_0$ but not in the Lie algebra of $\mathrm{U}_i$. We denote by $\check{\mathfrak{a}}^{\mathrm{P}_i}_{\mathrm{P}_0}$ the subspace of $ \check{\mathfrak{a}}_{\mathrm{P}_0}$ generated by the elements in $\Delta^{\mathrm{P}_i}_{\mathrm{P}_0}$. Let $\mathfrak{a}^{\mathrm{P}_i}_{\mathrm{P}_0}$ be the subspace of $\mathfrak{a}_{\mathrm{P_0}}$ annihilated by $\check{\mathfrak{a}}_{\mathrm{P_i}}$. Let $\mathrm{A}^{\mathrm{P}_i}_{\mathrm{P}_0} \subset \mathrm{A}_{\mathrm{P}_0}$ be the subtorus whose corresponding Lie subalgebra is $\mathfrak{a}^{\mathrm{P}_i}_{\mathrm{P}_0} \subset \mathfrak{a}_{\mathrm{P}_0}$. Let $\Delta( \mathrm{P}_0, \mathrm{A}^{\mathrm{P}_i}_{\mathrm{P}_0})$ be the system of simple roots defined by the choice of minimal parabolic $\mathrm{P}_0$ and the torus $\mathrm{A}^{\mathrm{P}_i}_{\mathrm{P}_0}$.

On the other hand, for $i = 1$ or $2$, let $\Omega^{\mathrm{P}_i}(\mathfrak{a}_{\mathrm{P_0}})$ be the set of isomorphisms of $\mathfrak{a}_{\mathrm{P}_0}$ given by the restriction to $\mathfrak{a}_{\mathrm{P}_0}$ of an element of the Weyl group $\mathcal{W}$ and leaving the space $\mathfrak{a}_{\mathrm{P_i}}$ pointwise fixed. In our case, $\Omega^{\mathrm{P}_i}(\mathfrak{a}_{\mathrm{P_0}})$ has two elements, one is the identity and the other one will be denoted by $s_i$. For the cases we will work on, the fact that $s_i\in \mathcal{W}_i^{0}$  and $\ell(w)+\ell(s_i w) = d_i$ will be enough to describe $s_i$.

Finally, for $w \in \mathcal{W}^{\mathrm{P}_0}$, we denote
\[
	\Lambda_w^{\mathrm{P}_i} = \left.-w(\lambda + \rho)\right|_{\mathfrak{a}^{\mathrm{P}_i}_{\mathrm{P}_0}}
\]
Although in Section $6$ of ~\cite{Schwermer1994} one finds this definition with $\rho_{\mathrm{P}_0}$ (as in Section $1.7$ of ~\cite{Sch83}) instead of $\rho$, one has $\left.\rho\right|_{\mathfrak{a}_{\mathrm{P}_0}}=\rho_{\mathrm{P}_0}$ (see Section $1.7$ of ~\cite{Sch83}). That is why one also finds this definition with $\rho$ instead of $\rho_{\mathrm{P}_0}$ in the introduction of ~\cite{Schwermer1994}.
With all this notation, we can now introduce the theorem that we will use, whose details for the proof can be found in ~\cite{Sch83} and ~\cite{Schwermer1994}.

\begin{thm} \label{EisensteinCohomologyThm}
Let $i$ be $1$ or $2$. Let $w \in \mathcal{W}^{\mathrm{P}_0}$ be such that, if $w = w^{\mathrm{P}_i/\mathrm{P}_0}w^{\mathrm{P}_i}$ with respect to the decomposition $\mathcal{W}^{\mathrm{P}_0} = \mathcal{W}^0_i\mathcal{W}^{\mathrm{P}_i}$, then $\ell(w^{\mathrm{P}_i/\mathrm{P}_0}) \geq \frac{d_i}{2}$. Let $[\varphi]$ be a cohomology class in $Ind_{\mathrm{P}_0}^{\mathrm{G}} H^0(S^{\mathrm{M}_0}, \widetilde{W}_{w_\ast(\lambda)})$ represented by a cuspidal form $\phi$. Let $E(\varphi, \Lambda)$ be the Eisenstein series  in the complex variable $\Lambda$, defined formally in Section $6$ of ~\cite{Schwermer1994}. Then:
\begin{enumerate}
\item[(a)] If $\langle \Lambda_w^{\mathrm{P}_i}, \alpha^\vee\rangle  > \langle \left.\rho\right|_{\mathfrak{a}^{\mathrm{P}_i}_{\mathrm{P}_0}}, \alpha^\vee\rangle$ for all $\alpha \in \Delta(\mathrm{P}_0, \mathrm{A}_{\mathrm{P}_0}^{\mathrm{P}_i})$ (i.e. if $\Lambda_w^{\mathrm{P}_i} - \left.\rho\right|_{\mathfrak{a}^{\mathrm{P}_i}_{\mathrm{P}_0}}$ is in the positive Weyl chamber of the system of simple roots $\Delta(\mathrm{P}_0, \mathrm{A}_{\mathrm{P}_0}^{\mathrm{P}_i})$) then the Eisenstein series $E(\varphi, \Lambda)$ is holomorphic at $\Lambda = \Lambda_w^{\mathrm{P}_i}$.
\item[(b)] If $\langle \Lambda_w^{\mathrm{P}_i}, \alpha^\vee\rangle  > 0$ for all $\alpha \in \Delta(\mathrm{P}_0, \mathrm{A}_{\mathrm{P}_0}^{\mathrm{P}_i})$ (i.e. if $\Lambda_w^{\mathrm{P}_i}$ is in the positive Weyl chamber of the system of simple roots $\Delta(\mathrm{P}_0, \mathrm{A}_{\mathrm{P}_0}^{\mathrm{P}_i})$) and the highest weight $w^{\mathrm{P}_i}_\ast(\lambda)$ of $\mathrm{M}_i$ is regular, then the Eisenstein series $E(\varphi, \Lambda)$ is holomorphic at $\Lambda = \Lambda_w^{\mathrm{P}_i}$.
\end{enumerate}  
In both cases, $E(\varphi, \Lambda_w^{\mathrm{P}_i})$ defines a closed form representing a cohomology class $[E(\varphi, \Lambda_w^{\mathrm{P}_i})]$ in $Ind_{\mathrm{P}_i}^{\mathrm{G}} H^{\ell(w^{\mathrm{P}_i/\mathrm{P}_0})}(S^{\mathrm{M}_i}, \widetilde{W}_{(w^{\mathrm{P}_i})_\ast(\lambda)}) \subset H^{\ell(w)}(\partial_i, \tilde{V}_\lambda)$ and one has:
\begin{enumerate}
\item If $\ell(w^{\mathrm{P}_i/\mathrm{P}_0}) > \frac{d_i}{2}$ then $r_i([E(\varphi, \Lambda_w^{\mathrm{P}_i})]) = [\varphi]$.
\item If $\ell(w^{\mathrm{P}_i/\mathrm{P}_0}) = \frac{d_i}{2}$ then, let $w'$ be $(s_iw^{\mathrm{P}_i/\mathrm{P}_0})w^{\mathrm{P}_i}$. One has
\[
r_i([E(\varphi, \Lambda_w^{\mathrm{P}_i})]) = [\varphi] + c(\Lambda_w^{\mathrm{P}_i})[\varphi] \in Ind_{\mathrm{P}_0}^{\mathrm{G}} H^0(S^{\mathrm{M}_0}, \widetilde{W}_{w_\ast(\lambda)}) \oplus Ind_{\mathrm{P}_0}^{\mathrm{G}} H^0(S^{\mathrm{M}_0}, \widetilde{W}_{w'_\ast(\lambda)}) \subset H^{\ell(w)}(\partial_0, \tilde{V}_\lambda).
\]
where $c(\Lambda_w^{\mathrm{P}_i}):Ind_{\mathrm{P}_0}^{\mathrm{G}} H^0(S^{\mathrm{M}_0}, \widetilde{W}_{w_\ast(\lambda)}) \rightarrow Ind_{\mathrm{P}_0}^{\mathrm{G}} H^0(S^{\mathrm{M}_0}, \widetilde{W}_{w'_\ast(\lambda)}) \subset H^{\ell(w)}(\partial_0, \tilde{V}_\lambda)$ is certain intertwining operator (that will not be used in this paper).
\end{enumerate} 
\end{thm}
  
\begin{proof}
For the case $(a)$, when $\Lambda_w^{\mathrm{P}_i} - \left.\rho\right|_{\mathfrak{a}^{\mathrm{P}_i}_{\mathrm{P}_0}}$ is in the positive Weyl chamber of the system of simple roots $\Delta(\mathrm{P}_0, \mathrm{A}_{\mathrm{P}_0}^{\mathrm{P}_i})$, this theorem is a combination of results of Section $6$ in ~\cite{Schwermer1994}, in particular Theorem $6.3$, Theorem $6.4$ and the proposition of that section. We observe that the result enunciated is this theorem is true even for nonregular highest weight $\lambda$, because the fact that $\Lambda_w^{\mathrm{P}_i} - \left.\rho\right|_{\mathfrak{a}^{\mathrm{P}_i}_{\mathrm{P}_0}}$ is in the positive Weyl chamber of the system of simple roots $\Delta(\mathrm{P}_0, \mathrm{A}_{\mathrm{P}_0}^{\mathrm{P}_i})$ already implies that the Eisenstein series is holomorphic at $\Lambda_w^{\mathrm{P}_i}$ and represents a closed form in $H^{\ell(w)}(\partial_i, \tilde{V}_\lambda)$. Then we can use the same reasoning as in the proof of Theorem $6.4$ of ~\cite{Schwermer1994} and Theorem $4.11$ of ~\cite{Sch83} to get the description of $r_i([E(\varphi, \Lambda_w^{\mathrm{P}_i})])$. 

For the item $(1)$, in principle one has 
\[
r_i([E(\varphi, \Lambda_w^{\mathrm{P}_i})]) = [\varphi] + c(\Lambda_w^{\mathrm{P}_i})[\varphi] \in Ind_{\mathrm{P}_0}^{\mathrm{G}} H^0(S^{\mathrm{M}_0}, \widetilde{W}_{w_\ast(\lambda)}) \oplus Ind_{\mathrm{P}_0}^{\mathrm{G}} H^0(S^{\mathrm{M}_0}, \widetilde{W}_{(s_iw)_\ast(\lambda)}) 
\]
On the other hand $Ind_{\mathrm{P}_0}^{\mathrm{G}} H^0(S^{\mathrm{M}_0}, \widetilde{W}_{w_\ast(\lambda)}) \subset H^{\ell(w)}(\partial_0, \tilde{V}_\lambda)$ and therefore $[E(\varphi, \Lambda_w^{\mathrm{P}_i})] \in H^{\ell(w)}(\partial_i, \tilde{V}_\lambda)$.

But $\ell(s_iw^{\mathrm{P}_i/\mathrm{P}_0}) < \ell(w^{\mathrm{P}_i/\mathrm{P}_0})$ and therefore $\ell(s_iw) < \ell(w)$. Therefore $H^0(S^{\mathrm{M}_0}, \widetilde{W}_{(s_iw)_\ast(\lambda)})$ defines cohomology classes in degree $\ell(s_iw)$. Hence $c(\Lambda_w^{\mathrm{P}_i})[\varphi] = 0$.

For the case $(b)$, when $\Lambda_w^{\mathrm{P}_i}$ is in the positive Weyl chamber and the highest weigth $w^{\mathrm{P}_i}_\ast(\lambda)$ of $\mathrm{M}_i$ is regular, we still need to prove that the Eisenstein series $E(\varphi, \Lambda)$ is holomorphic at $\Lambda = \Lambda_w^{\mathrm{P}_i}$. For this we observe the following fact. One knows that, for $i \in \left\{0, 1, 2\right\}$ one has a decomposition 
\begin{equation}
H^q(\partial_i, \widetilde{V}_\lambda) = \bigoplus_{w \in \mathcal{W}^{\mathrm{P}_i}} Ind_{\mathrm{P}_i}^{\mathrm{G}} H^{q-\ell(w)}(S^{\mathrm{M}_i}, \widetilde{W}_{w_\ast(\lambda)}) \,. \nonumber
\end{equation}
If $i$ is $1$ or $2$, then for $w \in \mathcal{W}^{\mathrm{P}_i}$, such that $w = w^{\mathrm{P}_i/\mathrm{P}_0}w^{\mathrm{P}_i}$ with respect to the decomposition $\mathcal{W}^{\mathrm{P}_0} = \mathcal{W}^0_i\mathcal{W}^{\mathrm{P}_i}$, the restriction of $r_i$ to the summand
\begin{equation} \label{Levi}
Ind_{\mathrm{P}_i}^{\mathrm{G}} H^{q-\ell(w)}(S^{\mathrm{M}_i}, \widetilde{W}_{(w^{\mathrm{P}_i})_\ast(\lambda)})
\end{equation}
has image in
\begin{equation} \label{boundaryMi}
\bigoplus_{\tilde{w} \in \mathcal{W}_i^0} Ind_{\mathrm{P}_0}^{\mathrm{G}} H^{q-\ell(w)-\ell(\tilde{w})}(S^{\mathrm{M}_0}, \widetilde{W}_{(\tilde{w}w^{\mathrm{P}_i})_\ast(\lambda)})\, 
\end{equation}

and (\ref{boundaryMi}) can be thought of as the boundary of the Borel-Serre compactification of (\ref{Levi}). One could therefore think about the construction of Eisenstein cohomology classes in 
\[
Ind_{\mathrm{P}_i}^{\mathrm{G}} H^{q-\ell(w)}(S^{\mathrm{M}_i}, \widetilde{W}_{(w^{\mathrm{P}_i})_\ast(\lambda)})
\]
from cohomology classes in the space $Ind_{\mathrm{P}_0}^{\mathrm{G}} H^{q-\ell(w)}(S^{\mathrm{M}_0}, \widetilde{W}_{w_\ast(\lambda)})$ as in ~\cite{Sch83}. 
We remark here that the parabolic induction $Ind_{\mathrm{P}_1}^{\mathrm{G}}$ appears after taking the inverse limit over the level varieties. So, we can work on the level varieties and then, when taking the inverse limit,  obtain the same results for $Ind_{\mathrm{P}_i}^{\mathrm{G}} H^{q-\ell(w)}(S^{\mathrm{M}_i}, \widetilde{W}_{(w^{\mathrm{P}_i})_\ast(\lambda)})$ or we can use the exactness of the parabolic induction.

In that case we would be thinking about the reductive group $\mathrm{M_i}$ and the $\mathbb{Q}$-system of positive roots $\Phi^+_{\mathrm{M_i}} = \Phi^+(\mathrm{M_i}, \mathrm{A}_\mathrm{P_i})$ defined by the (minimal) $\mathbb{Q}$-parabolic subgroup $\mathrm{P}_0 \cap \mathrm{M_i}$. In this setting, one has the corresponding element $\rho_{\mathrm{M_i}} = \sum_{\alpha \in \Phi^+_{\mathrm{M_i}}} \alpha$. One can see that $\left.\rho_{\mathrm{M_i}}\right|_{\mathfrak{a}^{\mathrm{P}_i}_{\mathrm{P}_0}} = \left.\rho\right|_{\mathfrak{a}^{\mathrm{P}_i}_{\mathrm{P}_0}}$. 
In fact, remember that the evaluation point, as in Theorem $4.11$ of ~\cite{Sch83}, in this case is given by
\[
\left.-w^{\mathrm{P}_i/\mathrm{P}_0}((w^{\mathrm{P}_i})_\ast(\lambda) + \rho_{\mathrm{M_i}}) \right|_{\mathfrak{a}^{\mathrm{P}_i}_{\mathrm{P}_0}}  =  \left.-w^{\mathrm{P}_i/\mathrm{P}_0}((w^{\mathrm{P}_i}(\lambda+\rho) - \rho) + \rho_{\mathrm{M_i}}) \right|_{\mathfrak{a}^{\mathrm{P}_i}_{\mathrm{P}_0}} = \left.-w(\lambda + \rho)\right|_{\mathfrak{a}^{\mathrm{P}_i}_{\mathrm{P}_0}}
\] 
Then, as it is already explained in the proof of Theorem $6.3$ of ~\cite{Schwermer1994}, one has that if the highest weight $(w^{\mathrm{P}_i})_\ast(\lambda)$ for $\mathrm{M}_i$ is regular, the Eisenstein series $E(\varphi, \Lambda)$ does not have a pole at $\Lambda_w^{\mathrm{P}_i}$, otherwise the residue of that Eisenstein series would represent a square integrable cohomology class in $Ind_{\mathrm{P}_i}^{\mathrm{G}} H^{q-\ell(w)}(S^{\mathrm{M}_i}, \widetilde{W}_{(w^{\mathrm{P}_i})_\ast(\lambda)})$ (see the comment before Proposition in Section $6$ of \cite{Schwermer1994}). But in the regular case, the square integrable cohomology is equal to the cuspidal cohomology (Corollary $2.3$ in  ~\cite{Schwermer1994}). This would be a contradiction, since the Eisenstein series could not represent cuspidal cohomology classes.
\end{proof}

Let $l$ be the rank of $G_n$, as defined in Subsections ~\ref{rsystemodd} and ~\ref{rsystemeven}. In the case treated in this paper one has that $\rho$ is given by
\[
	\rho = \left\{ \begin{array}{ccc}
 \sum_{k=1}^\ell (l-k+\frac{1}{2}) \epsilon_k, & \mbox{if $n$ is odd} \\
 \sum_{k=1}^\ell (l-k) \epsilon_k, & \mbox{if $n$ is even} \end{array} \right..	
\]

\subsubsection{The case $i=1$}

In this case, by using the results in Section \ref{wr} one has $d_i = \left|\Delta_i\right| - \left|\Delta_0\right|$, then
\[ 
d_1 = \left|\Delta_1\right| - \left|\Delta_0\right| = \left\{ \begin{array}{ccc}
 1 + 2(l-2), & \mbox{if $n$ is odd} \\
 2(l-2), & \mbox{if $n$ is even} \end{array} \right..
 \]
where $\left|\Delta_i\right|$ denotes the cardinality of the set $\Delta_i$. 

$\Delta^{\mathrm{P}_1}_{\mathrm{P}_0} \subset \Delta_\mathbb{Q} = \left\{ \varepsilon_1 - \varepsilon_2 \right\}$ and $\check{\mathfrak{a}}^{\mathrm{P}_1}_{\mathrm{P}_0}$ is the $\mathbb{R}$-vector space generated by $\varepsilon_1 - \varepsilon_2$. On the other hand, $\mathfrak{a}^{\mathrm{P}_1}_{\mathrm{P}_0}$ is generated by the character
\[
r \mapsto \left[ \begin{array}{ccccc}
 1 & 0 & 0 & 0 & 0\\
 0 & r & 0 & 0 & 0\\
 0 & 0 & Id_{n-2} & 0 & 0 \\ 
 0 & 0 & 0 & r^{-1} & 0 \\
 0 & 0 & 0 & 0 & 1 \end{array} \right]
\] 
or equivalently, under the natural isomorphism, it is the real vector subspace of $Lie(\mathrm{A}_{\mathrm{P}_0})$ generated by $E_{2,2}-E_{n+1,n+1}$ (where $E_{i, j}$ denotes the $(n+2) \times (n+2)$ matrix with $(i, j)$ entry $1$ and all other entries equal to $0$). 

\subsubsection{The case $i=2$}  
  
In this case, one has $d_2 = \left|\Delta_2\right| - \left|\Delta_0\right| = 1$.
  
$\Delta^{\mathrm{P}_2}_{\mathrm{P}_0} \subset \Delta_\mathbb{Q} = \left\{ \varepsilon_2 \right\}$ and $\check{\mathfrak{a}}^{\mathrm{P}_2}_{\mathrm{P}_0}$ is the $\mathbb{R}$-vector space generated by $\varepsilon_2$. On the other hand, $\mathfrak{a}^{\mathrm{P}_2}_{\mathrm{P}_0}$ is generated by the character
\[
r \mapsto \left[ \begin{array}{ccccc}
 r & 0 & 0 & 0 & 0\\
 0 & r^{-1} & 0 & 0 & 0\\
 0 & 0 & Id_{n-2} & 0 & 0 \\ 
 0 & 0 & 0 & r & 0 \\
 0 & 0 & 0 & 0 & r^{-1} \end{array} \right] 
\] 
or equivalently, under the natural isomorphism, it is the real vector subspace of $Lie(\mathrm{A}_{\mathrm{P}_0})$ generated by $E_{1,1}-E_{2,2}+E_{n+1,n+1}-E_{n+2,n+2}$.   

\begin{thm} \label{GL2}
If $w = w^{\mathrm{P}_2/\mathrm{P}_0}w^{\mathrm{P}_2} \in \mathcal{W}^{\mathrm{P}_0} = \mathcal{W}^0_2\mathcal{W}^{\mathrm{P}_2}$ and $(w^{\mathrm{P}_2})_\ast(\lambda) = n_1\epsilon_1 + \ldots + n_l\epsilon_l +c\kappa$. Then, if $w^{\mathrm{P}_2/\mathrm{P}_0} = w_{(1, 2), \mathfrak{1}}$ (notation as in ~\ref{weylodd} and ~\ref{weyleven}) and $n_1 > n_2$ then the space $Ind_{\mathrm{P}_0}^{\mathrm{G}} H^{0}(S^{\mathrm{M}_0}, \widetilde{W}_{w_\ast(\lambda)})$ is contained in the image of $r_2$ and does not contribute to ghost classes.
\end{thm}

\begin{proof}
We know that $d_2=1$, therefore, under the hypothesis of the theorem $\ell(w^{\mathrm{P}_2/\mathrm{P}_0}) > d_2$. On the other hand,  $w_\ast(\lambda) = w^{\mathrm{P}_2/\mathrm{P}_0}_\ast((w^{\mathrm{P}_i})_\ast(\lambda)) = w_{(1, 2), \mathfrak{1}}((w^{\mathrm{P}_i})_\ast(\lambda) + \rho) - \rho$.
Therefore if $w_\ast(\lambda) = m_1\epsilon_1 + \ldots + m_l\epsilon_l +c\kappa$ one has $m_1=n_2-1$ and $m_2=n_1+1$. Moreover, $\Lambda_w^{\mathrm{P}_2} = \left.-w(\lambda+\rho)\right|_{\mathfrak{a}^{\mathrm{P}_i}_{\mathrm{P}_0}} = \left.-(w_\ast(\lambda)+\rho)\right|_{\mathfrak{a}^{\mathrm{P}_i}_{\mathrm{P}_0}}$. Then the inequality in item $(a)$ of Theorem \ref{EisensteinCohomologyThm} is given by $-(m_1-m_2) > 2$, but this means $2 + (n_1-n_2) > 2$. Therefore, if $n_1 > n_2$ then the hypothesis of items $(a)$ and $(1)$ of Theorem \ref{EisensteinCohomologyThm} are satisfied and the result is proved. This theorem can also be proved by using Theorem $2$ in \cite{Harder87} together with the exactness of the parabolic induction.
\end{proof}

\section{Ghost Classes For $\mathrm{GO}(2, 4)$}\label{GO24}
In this section, we closely study each element $w \in \mathcal{W}^{\mathrm{P}_0}$ to determine when the associated space $Ind_{\mathrm{P}_0}^{\mathrm{G}} H^{0}(S^{\mathrm{M}_0},  \widetilde{W}_{w_\ast(\lambda)})$ will have possible contribution to ghost classes. This is done by using the discussion carried out in Section~\ref{wr}, Section~\ref{mht} and the facts listed in Section~\ref{facts}. In this case the set of Weyl representatives $\mathcal{W}^{\mathrm{P}_0}$ is the whole Weyl group $\mathcal{W}$.  In this particular case, the description of the sets of Weyl representatives given in the Subsection~\ref{weyleven}  can be summarized as follows:
\begin{itemize}
\item $\mathcal{W}^{\mathrm{P}_0} = \mathcal{W}$, this is the set of all $24$ elements listed in Table~\ref{table2} below.
\item  $\mathcal{W}^{\mathrm{P}_2} =  \left\{ w_{1},  w_{4},  w_{6},  w_{8}, w_{9},  w_{11},  w_{13},  w_{14}, w_{15}, w_{16}, w_{17}, w_{18} \right\} .$
\item $\mathcal{W}^{\mathrm{P}_1} = \left\{ w_{1},  w_{2},  w_{5},  w_{19}, w_{20},  w_{23}\right\}$.
\item $\mathcal{W}_1^{0} = \left\{  w_{1}, w_{4},  w_{13}, w_{16}\right\}$.
\item $\mathcal{W}_2^{0} = \left\{ w_{1}, w_{2}\right\}$.
\end{itemize} 
We present a table with the elements in $\mathcal{W}^{\mathrm{P}_0}$ and where each column delivers specific information described below.
\begin{table}[ht] \label{weights4}
\caption{The set of Weyl representatives $\mathcal{W}^{\mathrm{P}_0}$ for $GO(2,4)$}
\label{table2}
\centering
\begin{tabular}{lllccllrrr}
\hline \noalign{\smallskip}
$w$ & $\sigma$ & $f$ & $\ell(w)$ & weight + $2c$ & $\mathcal{W}_2^0 \mathcal{W}^{\mathrm{P}_2}$ & $\mathcal{W}_1^0 \mathcal{W}^{\mathrm{P}_1}$ & $n_1$ & $n_2$ & $n_3$ \\
\noalign{\smallskip}\hline\noalign{\smallskip}
$w_1$ & $e$ & $\emptyset$ & $0$ &  $-2a_1$  & $w_1 w_1$ & $w_1 w_1$ & $a_1$ & $a_2$ & $a_3$\\ 
$w_{2}$ & $(1\, 2)$ & $\emptyset$ & $1$ & $2-2a_2$ & $w_2 w_1$ & $w_1 w_2$ & $a_2-1$ & $a_1+1$  & $a_3$\\ 
$w_{3}$ & $(1\, 3)$ & $\emptyset$ & $3$ & $4-2a_3$ & $w_2 w_6$  & $w_4 w_5$ & $a_3-2$ & $a_2$ & $a_1+2$\\ 
$w_{4}$ & $(2\, 3)$ & $\emptyset$ & $1$ &  $-2a_1$  & $w_1 w_4$  & $w_4 w_1$ & $a_1$ & $a_3-1$ & $a_2 +1$ \\ 
$w_{5}$ & $(1\, 2\, 3)$ & $\emptyset$ & $2$ & $4-2a_3$ & $w_2 w_4$ & $w_1 w_5$ & $a_3 - 2$ & $a_1 + 1$ & $a_2 +1$ \\ 
$w_{6}$ & $(3\, 2\, 1)$ & $\emptyset$ & $2$ & $2-2a_2$ & $w_1 w_6$ & $w_4 w_2$ & $a_2 - 1$ & $a_3 - 1$ & $a_1+2$ \\ 
$w_{7}$ & $e$ & $\left\{1, 2 \right\}$ & $6$ & $8+2a_1$ & $w_2 w_{8}$ & $w_{13} w_{19}$ & $-a_1-4$ & $-a_2 - 2$ & $a_3$ \\
$w_{8}$ & $(1\, 2)$ & $\left\{1, 2 \right\}$ & $5$ & $6+2a_2$ & $w_1 w_{8}$ & $w_{13} w_{20}$ & $-a_2-3$ & $-a_1-3$ & $a_3$ \\ 
$w_{9}$ & $(1\, 3)$ & $\left\{1, 2 \right\}$ & $3$ & $4+2a_3$ & $w_1 w_{9}$ & $w_{4} w_{23}$  & $-a_3-2$ & $-a_2-2$ & $a_1+2$\\ 
$w_{10}$ & $(2\, 3)$ & $\left\{1, 2 \right\}$ & $5$ & $8+2a_1$ & $w_2 w_{11}$ & $w_{4} w_{19}$& $-a_1-4$ & $-a_3-1$ & $a_2+1$\\ 
$w_{11}$ & $(1\, 2\, 3)$ & $\left\{1, 2 \right\}$ & $4$ & $4+2a_3$ & $w_1 w_{11}$ &$w_{13} w_{23}$ & $-a_3-2$ & $-a_1 - 3$ &$a_2+1$ \\ 
$w_{12}$ & $(3\, 2\, 1)$ &$\left\{1, 2 \right\}$ & $4$ & $6+2a_2$ & $w_2 w_{9}$ & $w_{4} w_{20}$& $- a_2- 3$ & $-a_3 - 1$ & $a_1+2$\\  
$w_{13}$ & $e$ & $\left\{2, 3 \right\}$ & $2$ & $-2a_1$ & $w_1 w_{13}$ & $w_{13} w_{1}$ & $a_1$ & $-a_2-2$ & $-a_3$\\ 
$w_{14}$ & $(1\, 2)$ & $\left\{2, 3 \right\}$ & $3$ & $2-2a_2$ & $w_1 w_{14}$ & $w_{13} w_{2}$ & $a_2-1$ & $-a_1-3$ &$-a_3$ \\ 
$w_{15}$ & $(1\, 3)$ & $\left\{2, 3 \right\}$ & $3$ & $4-2a_3$ & $w_1 w_{15}$ & $w_{16} w_{5}$& $a_3-2$ & $-a_2-2$ & $- a_1 -2$\\
$w_{16}$ & $(2\, 3)$ & $\left\{2, 3 \right\}$ & $1$ & $-2a_1$ & $w_1 w_{16}$ & $w_{16} w_{1}$& $a_1$ & $-a_3-1$ & $-a_2-1$\\  
$w_{17}$ & $(1\, 2\, 3)$ & $\left\{2, 3 \right\}$ & $4$ & $4-2a_3$ & $w_1 w_{17}$ & $w_{13} w_{5}$ & $a_3 - 2$ & $- a_1 -3$ &$-a_2-1$ \\ 
$w_{18}$ & $(3\, 2\, 1)$ & $\left\{2, 3 \right\}$ & $2$ & $2-2a_2$ & $w_1 w_{18}$ & $w_{16} w_{2}$& $a_2 - 1$ & $-a_3 - 1$ & $- a_1 -2$\\ 
$w_{19}$ & $e$ & $\left\{1, 3 \right\}$ & $4$ & $8+2a_1$ & $w_2 w_{14}$  & $w_{1} w_{19}$ & $-a_1-4$ & $a_2$ & $-a_3$\\ 
$w_{20}$ & $(1\, 2)$ & $\left\{1, 3 \right\}$ & $3$ & $6+2a_2$ & $w_2 w_{13}$ & $w_{1} w_{20}$& $-a_2-3$ & $a_1 +1$ & $-a_3$\\
$w_{21}$ & $(1\, 3)$ & $\left\{1, 3 \right\}$ & $3$ & $4+2a_3$ & $w_2 w_{18}$ &$w_{16} w_{23}$ & $-a_3-2$ & $a_2$ & $- a_1 -2$\\ 
$w_{22}$ & $(2\, 3)$ & $\left\{1, 3 \right\}$ & $5$ & $8+2a_1$ & $w_2 w_{17}$ & $w_{16} w_{19}$& $-a_1-4$ & $a_3-1$ & $-a_2-1$\\ 
$w_{23}$ & $(1\, 2\, 3)$ & $\left\{1, 3 \right\}$ & $2$ & $4+2a_3$ & $w_2 w_{16}$ &$w_{1} w_{23}$ & $-a_3 - 2$ & $a_1+1$ & $-a_2-1$\\ 
$w_{24}$&$(3\, 2\, 1)$ & $\left\{1, 3 \right\}$ & $4$ & $6+2a_2$ & $w_2 w_{15}$ & $w_{16} w_{20}$ & $- a_2 - 3$ & $a_3 - 1$ & $- a_1 -2$\\ 
\noalign{\smallskip}\hline
\end{tabular}
\end{table}

In the first column of Table~\ref{table2}, we indicate the Weyl representatives determined by the permutation $\sigma \in S_3$ and the choice of signs $f$ given in the second and third column respectively. In the third column we describe $f$  by giving the set $f^{-1}(-1) \subset \left\{1, 2, 3 \right\}$. The fourth column collects the length of the corresponding Weyl representative and the fifth column indicates the weights in the mixed Hodge structure of $Ind_{\mathrm{P}_0}^{\mathrm{G}} H^0(S^{\mathrm{M}_0}, \widetilde{W}_{w_\ast(\lambda)})$ plus $2c$ (this is just $-2n_1$, by Subsection ~\ref{weight0}). The sixth and seventh column indicates the components of $w$ with respect to the decomposition $\mathcal{W}_2^0 \mathcal{W}^{\mathrm{P}_2}$  and $\mathcal{W}_1^0 \mathcal{W}^{\mathrm{P}_1}$ of $\mathcal{W}^{\mathrm{P}_0}$. In the last three columns we write the coefficients $n_1, n_2, n_3$ from the expression $w_\ast(\lambda) = n_1 \epsilon_1 + n_2 \epsilon_2 + n_3 \epsilon_3 + c \kappa$. We now prove the following 

\begin{thm} \label{Thm2}
Let $V_\lambda$ be the finite dimensional irreducible representation of $\mathrm{GO}(2, 4)$ with highest weight $\lambda = a_1 \epsilon_1 + a_2 \epsilon_2 + a_3 \epsilon_3 + c \kappa$. One has:
\begin{enumerate}
\item If $a_2 \neq 0$, then there are no ghost classes in the cohomology space $H^\bullet(\partial \overline{S}, \V_\lambda)$.

\item If $a_2 = 0$ (which implies $a_3=0$ and therefore, in terms of fundamental weights, $\lambda = a_1 \varpi_1 + c \kappa$), then the only possible weights in the space of ghost classes are the middle weight and the middle weight plus one.
\end{enumerate}
\end{thm}

\begin{proof}
We begin by using the facts from Subsection~\ref{MW} to eliminate certain possible contributions of the spaces $Ind_{\mathrm{P}_0}^{\mathrm{G}} H^0(S^{\mathrm{M}_0}, \widetilde{W}_{w_\ast(\lambda)})$ to ghost classes for $w\in \mathcal{W}^{\mathrm{P}_0}$. Following Lemma~\ref{mwelimination},  one can see by comparing the entries of fourth and fifth columns of Table~\ref{table2} that for the Weyl representatives
\[
w \in \left\{ w_1, w_4, w_6, w_{13}, w_{14},  w_{16}, w_{18}\right\}
\]

the space $Ind_{\mathrm{P}_0}^{\mathrm{G}} H^0(S^{\mathrm{M}_0}, \widetilde{W}_{w_\ast(\lambda)})$ has weight less than the middle weight of $H^{\ell(w)+1}(S, \V_\lambda)$ and therefore this space cannot contribute to ghost classes. In addition, we note the following
\begin{enumerate}
\item[(i)] For $w=w_2$, the space $Ind_{\mathrm{P}_0}^{\mathrm{G}} H^0(S^{\mathrm{M}_0}, \widetilde{W}_{w_\ast(\lambda)})$ could only contribute to ghost classes if $a_2=0$.
\item[(ii)] For $w \in \{w_3, w_5, w_{15}\}$, the space $Ind_{\mathrm{P}_0}^{\mathrm{G}} H^0(S^{\mathrm{M}_0}, \widetilde{W}_{w_\ast(\lambda)})$ could only contribute to ghost classes if $a_3 \leq 0$.
\item[(iii)] For $w \in \{ w_9, w_{21}, w_{23}\}$, the space $Ind_{\mathrm{P}_0}^{\mathrm{G}} H^0(S^{\mathrm{M}_0}, \widetilde{W}_{w_\ast(\lambda)})$ could only contribute to ghost classes if $a_3 \geq 0$.
\end{enumerate} 

Following Theorem ~\ref{GL2}, we study the image of the morphism
\[
r_2:H^\bullet(\partial_2, \widetilde{V}_\lambda) \rightarrow H^\bullet(\partial_0, \widetilde{V}_\lambda)\,.
\]
For $w \in \mathcal{W}^{\mathrm{P}_0} $ we write by $w = w^{\mathrm{P}_2/\mathrm{P}_0}w^{\mathrm{P}_2} \in \mathcal{W}_2^0 \mathcal{W}^{\mathrm{P}_2}$ its decomposition described in the sixth column of Table~\ref{table2}. We see that for
\[
w \in \left\{ w_5, w_{10}, w_{19}, w_{20}, w_{22}, w_{23}\right\}\,.
\]
the component in $\mathcal{W}_2^0$ is $w_2$ and the component in $\mathcal{W}^{\mathrm{P}_2}$ is, respectively, $w_4, w_{11}, w_{14}, w_{13}, w_{17}$ and $w_{16}$. For each of these $w^{\mathrm{P}_2}$, we can see, following the values of $n_1$ and $n_2$ in the expression $(w^{\mathrm{P}_2})_\ast(\lambda) = n_1\epsilon_1 + \ldots + n_l\epsilon_l +c\kappa$ encoded in the last two columns of the Table~\ref{table2}, that $n_1 > n_2$. By Theorem ~\ref{GL2}, this implies that the associated space $Ind_{\mathrm{P}_0}^{\mathrm{G}} H^0(S^{\mathrm{M}_0}, \widetilde{W}_{w_\ast(\lambda)})$ will be entirely contained in the image of $r_2$ and therefore this space cannot contribute to ghost classes. 

Using the same argument for the cases $w_2$ and $w_7$, the corresponding space can contribute to ghost classes only when $a_1=a_2$. For $w_3$ and $w_{12}$, the corresponding space could contribute to ghost classes only when $a_2=a_3$ and the cases $w_{21}$ and $w_{24}$ could contribute to ghost classes only when $a_2=-a_3$.

Note that the case $w=w_7$ could only contribute to ghost classes in degree $7$. As the dimension of the symmetric space associated to $\mathrm{G}$ is $8$, then by Corollary $11.4.3$ in \cite{BoSe73} one can rule out the possibility of contribution to ghost classes.

Now, we continue analyzing further the possible contribution of the space $Ind_{\mathrm{P}_0}^{\mathrm{G}} H^0(S^{\mathrm{M}_0}, \widetilde{W}_{w_\ast(\lambda)})$ for the remaining Weyl representatives, \ie for  
\begin{equation}\label{eq:rw24}
w\in \{ w_2, w_3, w_8, w_9, w_{11}, w_{12}, w_{15}, w_{17}, w_{21}, w_{24} \} \,,
\end{equation}
by studying the image of the restriction of the map $r_1:H^\bullet(\partial_{\mathrm{P}_1}, \V_\lambda) \rightarrow H^\bullet(\partial_{\mathrm{P}_0}, \V_\lambda)$ following the discussion of Subsection ~\ref{r1}.

One has $s_1=w_{13}$, $w_\ast(\lambda) = n_1 \epsilon_1 + n_2 \epsilon_2 +n_3 \epsilon_3 + c \kappa$, $\rho=2\epsilon_1 + \epsilon_2$, $\mathrm{A}_{\mathrm{P}_0}^{\mathrm{P}_1} =\{ \epsilon_2\}$, $d_1=2$ and $-w(\lambda+\rho)=-(w_\ast(\lambda)+\rho)$. Then, in this case, $\Lambda_w^{\mathrm{P}_1}= -(n_2+1)$ and the inequality in item $(a)$ of Theorem ~\ref{EisensteinCohomologyThm} is given by $-n_2 > 2$.

We see that for $w \in \{w_8, w_{11}, w_{17}\}$, its component in $\mathcal{W}^{\mathrm{P}_1/\mathrm{P_0}}$ with respect to the decompostion $\mathcal{W}^{\mathrm{P_0}}=\mathcal{W}^{\mathrm{P}_1/\mathrm{P_0}} \mathcal{W}^{\mathrm{P}_1}$ is $w_{13}$ ($\ell(w_{13}) = 2 > \frac{d_1}{2}$) and $w$ satisfies the condition $-n_2 > 2$. Thus, a direct application of item $(1)$ of Theorem ~\ref{EisensteinCohomologyThm} gives that $Ind_{\mathrm{P}_0}^{\mathrm{G}} H^0(S^{\mathrm{M}_0}, \widetilde{W}_{w_\ast(\lambda)})$ is contained in $Im(r_1)$ and therefore it does not contribute to ghost classes.

Again, in the setting of Theorem ~\ref{EisensteinCohomologyThm}, $w = w_{9} = w_{4}w_{23} \in \mathcal{W}^{\mathrm{P}_1/\mathrm{P_0}} \mathcal{W}^{\mathrm{P}_1}$, one has $\ell(w_{4}) = \frac{d_1}{2}$ and $\Lambda_w^{\mathrm{P}_1} = a_2+1$ and the inequality of item $(a)$ is given by $a_2+1 > 1$. Therefore, if $a_2 > 0$, all the hypothesis of item $(a)$ of the aforementioned theorem are satisfied. Thus, for every form $[\varphi] \in Ind_{\mathrm{P}_0}^{\mathrm{G}} H^0(S^{\mathrm{M}_0}, \widetilde{W}_{w_\ast(\lambda)})$, the projection to the first coordinate of 
\[
r_1([E(\varphi, \Lambda_w^{\mathrm{P}_i})]) = [\varphi] + c(\Lambda_w^{\mathrm{P}_i})[\varphi] \in Ind_{\mathrm{P}_0}^{\mathrm{G}} H^0(S^{\mathrm{M}_0}, \widetilde{W}_{w_\ast(\lambda)}) \oplus Ind_{\mathrm{P}_0}^{\mathrm{G}} H^0(S^{\mathrm{M}_0}, \widetilde{W}_{(w_{21})_\ast(\lambda)}) \subset H^{\ell(w)}(\partial_0, \tilde{V}_\lambda)
\]
is again $[\varphi]$. This together with the fact that, for $a_2 > 0$, $Ind_{\mathrm{P}_0}^{\mathrm{G}} H^0(S^{\mathrm{M}_0}, \widetilde{W}_{(w_{21})_\ast(\lambda)})$ is in the image of $r_2$ implies that $Ind_{\mathrm{P}_0}^{\mathrm{G}} H^0(S^{\mathrm{M}_0}, \widetilde{W}_{w_\ast(\lambda)})$ is contained in $Im(r_1) + Im(r_2)$. In conclusion, $w_9$ could contribute to ghost classes only if $a_2 = 0$. By the same procedure and using the already proved fact that $w_{15}$ can only contribute to ghost classes if $a_3 \leq 0$, one can show that $w_{15}$ can contribute to ghots classes only if $a_2=0$.

Finally, we make use of item $(b)$ in Theorem \ref{EisensteinCohomologyThm}. For $w \in \mathcal{W}^{\mathrm{P}_1}$, the highest weight  $w_\ast(\lambda) = n_1 \epsilon_1 + n_2 \epsilon_2 +n_3 \epsilon_3 + c \kappa$  is regular for $\mathrm{M}_1$ if $n_2 > \left|n_3\right| > 0$.

For $w = w_{12} = w_{4}w_{20} \in \mathcal{W}^{\mathrm{P}_1/\mathrm{P_0}} \mathcal{W}^{\mathrm{P}_1}$, we checked before that $w_{12}$ could contribute to ghost classes only if $a_3 = a_2$. Suppose $a_2 =a_3 > 0$. One has $\ell(w_{4}) = \frac{d_1}{2}$ and $\Lambda_w^{\mathrm{P}_1} = a_3+1$ and the inequality of item $(a)$ is given by $a_3+1 > 2$. On the other hand, $(w_{20})_\ast(\lambda)$ is regular. Therefore, if $a_2=a_3 > 0$, all the hypothesis of item $(b)$ of the aforementioned theorem are satisfied. Thus, for every form $[\varphi] \in Ind_{\mathrm{P}_0}^{\mathrm{G}} H^0(S^{\mathrm{M}_0}, \widetilde{W}_{w_\ast(\lambda)})$, the projection to the first coordinate of 
\[
r_1([E(\varphi, \Lambda_w^{\mathrm{P}_i})]) = [\varphi] + c(\Lambda_w^{\mathrm{P}_i})[\varphi] \in Ind_{\mathrm{P}_0}^{\mathrm{G}} H^0(S^{\mathrm{M}_0}, \widetilde{W}_{w_\ast(\lambda)}) \oplus Ind_{\mathrm{P}_0}^{\mathrm{G}} H^0(S^{\mathrm{M}_0}, \widetilde{W}_{(w_{24})_\ast(\lambda)}) \subset H^{\ell(w)}(\partial_0, \tilde{V}_\lambda).
\]
is again $[\varphi]$. This together with the fact that, for $a_3 > 0$, $Ind_{\mathrm{P}_0}^{\mathrm{G}} H^0(S^{\mathrm{M}_0}, \widetilde{W}_{(w_{24})_\ast(\lambda)})$ is in the image of $r_2$ implies that $Ind_{\mathrm{P}_0}^{\mathrm{G}} H^0(S^{\mathrm{M}_0}, \widetilde{W}_{w_\ast(\lambda)})$ is contained in $Im(r_1) + Im(r_2)$. In conclusion, $w_{12}$ could contribute to ghost classes only if $a_2 = a_3 = 0$. By the same procedure and using the already proved fact that $w_{24}$ can only contribute to ghost classes if $a_3 = -a_2$, one can show that $w_{24}$ could contribute to ghost classes only if $a_2=0$.

We now summarize the above discussion to point out the possible contribution of the spaces $Ind_{\mathrm{P}_0}^{\mathrm{G}} H^0(S^{\mathrm{M}_0}, \widetilde{W}_{{w}_\ast(\lambda)})$ to the ghost classes, as follows:

\begin{enumerate}
\item If $a_1 = a_2 = a_3 = 0$ then the space $Ind_{\mathrm{P}_0}^{\mathrm{G}} H^0(S^{\mathrm{M}_0}, \widetilde{W}_{{(w_2)}_\ast(\lambda)})$ could contribute to ghost classes in degree $2$ and would have weight equal to the middle weight of $H^2(S, \widetilde{V}_\lambda)$.

\item If $a_2 = a_3 =0$ then the space $Ind_{\mathrm{P}_0}^{\mathrm{G}} H^0(S^{\mathrm{M}_0}, \widetilde{W}_{{(w_3)}_\ast(\lambda)})$ could contribute to ghost classes in degree $4$ and would have weight equal to the middle weight of $H^4(S, \widetilde{V}_\lambda)$.

\item If $a_2 = a_3 = 0$ then the space $Ind_{\mathrm{P}_0}^{\mathrm{G}} H^0(S^{\mathrm{M}_0}, \widetilde{W}_{{(w_9)}_\ast(\lambda)})$ could contribute to ghost classes in degree $4$ and would have weight equal to the middle weight of $H^4(S, \widetilde{V}_\lambda)$.

\item If $a_2=a_3=0$ then the space $Ind_{\mathrm{P}_0}^{\mathrm{G}} H^0(S^{\mathrm{M}_0}, \widetilde{W}_{{(w_{12})}_\ast(\lambda)})$ could contribute to ghost classes in degree $5$ and would have weight equal to the middle weight of $H^5(S, \widetilde{V}_\lambda)$ plus one.

\item If $a_2=a_3=0$ then the space $Ind_{\mathrm{P}_0}^{\mathrm{G}} H^0(S^{\mathrm{M}_0}, \widetilde{W}_{{(w_{15})}_\ast(\lambda)})$ could contribute to ghost classes in degree $4$ and would have weight equal to the middle weight of $H^4(S, \widetilde{V}_\lambda)$.

\item If $a_2=a_3=0$ then the space $Ind_{\mathrm{P}_0}^{\mathrm{G}} H^0(S^{\mathrm{M}_0}, \widetilde{W}_{{(w_{21})}_\ast(\lambda)})$ could contribute to ghost classes in degree $4$ and would have weight equal to the middle weight of $H^4(S, \widetilde{V}_\lambda)$.

\item If $a_2=a_3=0$ then the space $Ind_{\mathrm{P}_0}^{\mathrm{G}} H^0(S^{\mathrm{M}_0}, \widetilde{W}_{{(w_{24})}_\ast(\lambda)})$ could contribute to ghost classes in degree $5$ and would have weight equal to the middle weight of $H^5(S, \widetilde{V}_\lambda)$ plus one.
\end{enumerate}
Hence, we have proved the theorem.

\end{proof}

We conclude the discussion with the following

\begin{coro} \label{Coro4}
Let $V_\lambda$ be the finite dimensional irreducible representation of $\mathrm{GO}(2, 4)$ with highest weight $\lambda = n_1 \varpi_1 + n_2 \varpi_2 + n_3 \varpi_3 + c \kappa$. Then ghost classes can exist only if $n_2=n_3=0$, and in that case one has
\begin{enumerate}
\item If $n_1 \neq 0$ then ghost classes can exist only in degree $4$ with middle weight and in degree $5$ with middle weight plus one.

\item If $n_1 = 0$ then ghost classes can exist only in degrees $2$ and $4$ with middle weight and in degree $5$ with middle weight plus one.
\end{enumerate}
\end{coro}

\section{Ghost Classes For $\mathrm{GO}(2, 5)$}\label{GO25}

In this last section, we will study each element $w \in \mathcal{W}^{\mathrm{P}_0}$  to determine when the associated space $Ind_{\mathrm{P}_0}^{\mathrm{G}} H^{0}(S^{\mathrm{M}_0},  \widetilde{W}_{w_\ast(\lambda)})$ will have possible contribution to ghost classes. Note that in this case the Weyl group of $\mathrm{P}_0$ is not the whole Weyl group of the underlying group.  In this case $\mathcal{W}^{\mathrm{P}_0}$ has $4l(l-1)=4\cdot 3\cdot 2 = 24$ elements. We proceed in a similar fashion as in Section~\ref{GO24} by using the results discussed in Section~\ref{wr}, Section~\ref{mht} and the facts listed in Section~\ref{facts}.

In this particular case, the description of the sets of Weyl representatives given in the Subsection~\ref{weylodd}  can be summarized as follows:
\begin{itemize}
\item $\mathcal{W}^{\mathrm{P}_0}\subset \mathcal{W}$, and all elements of $ \mathcal{W}^{\mathrm{P}_0}$  are listed in the first column of the Table~\ref{table1} below.
\item  $\mathcal{W}^{\mathrm{P}_2} =  \left\{ w_{1},  w_{4},  w_{6},  w_{13}, w_{14},  w_{15},  w_{16}, w_{17},  w_{18}, w_{20}, w_{21}, w_{23} \right\} .$

\item $\mathcal{W}^{\mathrm{P}_1} = \left\{ w_{1},  w_{2},  w_{5},  w_{7}, w_{8},  w_{11}\right\}$.

\item $\mathcal{W}_1^{0} = \left\{  w_{1}, w_{4},  w_{13}, w_{16}\right\}$.
\item $\mathcal{W}_2^{0} = \left\{ w_{1}, w_{2}\right\}$.
\end{itemize}

We present a similar table as provided in Section~\ref{GO24}, with the elements in $\mathcal{W}^{\mathrm{P}_0}$ and where each column delivers same type of information.

\begin{table}[ht] \label{weights5}
\caption{The set of Weyl representatives $\mathcal{W}^{\mathrm{P}_0}$ for $\mathrm{GO}(2,5)$}
\label{table1}
\centering
\begin{tabular}{cllccllrrr}
\hline \noalign{\smallskip}
$w$ & $\sigma$ & $f$ & $\ell(w)$ & weight + $2c$ & $\mathcal{W}_2^0 \mathcal{W}^{\mathrm{P}_2}$ &  $\mathcal{W}_1^0 \mathcal{W}^{\mathrm{P}_1}$ & $n_1$ & $n_2$ & $n_3$ \\
\noalign{\smallskip}\hline\noalign{\smallskip}
$w_1$ & $e$ & $\emptyset$ & $0$ &  $-2a_1$  & $w_1 w_1$ & $w_1 w_1$ & $a_1$ & $a_2$ & $a_3$ \\ 
$w_{2}$ & $(1\, 2)$ & $\emptyset$ & $1$ & $2-2a_2$ & $w_2 w_1$ & $w_1 w_2$ &$a_2-1$ & $a_1+1$ & $a_3$\\ 
$w_{3}$ & $(1\, 3)$ & $\emptyset$ & $3$ & $4-2a_3$ & $w_2 w_6$  & $w_4 w_5$ & $a_3-2$ & $a_2$ & $a_1+2$ \\ 
$w_{4}$ & $(2\, 3)$ & $\emptyset$ & $1$ &  $-2a_1$  & $w_1 w_4$  & $w_4 w_1$ & $a_1$ & $a_3-1$ & $a_2+1$\\ 
$w_{5}$ & $(1\, 2\, 3)$ & $\emptyset$ & $2$ & $4-2a_3$ & $w_2 w_4$ & $w_1 w_5$ & $a_3 - 2$ & $a_1 + 1$ & $a_2+1$ \\ 
$w_{6}$ & $(3\, 2\, 1)$ & $\emptyset$ & $2$ & $2-2a_2$ & $w_1 w_6$ & $w_4 w_2$ & $a_2 - 1$ & $a_3 - 1$ & $a_2+1$ \\ 
$w_{7}$ & $e$ & $\left\{1 \right\}$ & $5$ & $10+2a_1$ & $w_2 w_{14}$ & $w_1 w_7$ & $-a_1-5$ & $a_2$ & $a_3$ \\
$w_{8}$ & $(1\, 2)$ & $\left\{1 \right\}$ & $4$ & $8+2a_2$ & $w_2 w_{13}$ & $w_1 w_8$ & $-a_2-4$ & $a_1+1$ & $a_3$ \\ 
$w_{9}$ & $(1\, 3)$ & $\left\{1 \right\}$ & $4$ & $6+2a_3$ & $w_2 w_{18}$ & $w_4 w_{11}$ &$-a_3-3$ & $a_2$ & $a_1+2$\\ 
$w_{10}$ & $(2\, 3)$ & $\left\{1 \right\}$ & $6$ & $10+2a_1$ & $w_2 w_{17}$ & $w_4 w_{7}$ &$-a_1-5$ & $a_3-1$ & $a_2+1$\\ 
$w_{11}$ & $(1\, 2\, 3)$ & $\left\{1 \right\}$ & $3$ & $6+2a_3$ & $w_2 w_{16}$ & $w_1 w_{11}$ & $-a_3 - 3$ & $a_1 + 1$ & $a_2 + 1$\\ 
$w_{12}$ & $(3\, 2\, 1)$ & $\left\{1 \right\}$ & $5$ & $8+2a_2$ & $w_2 w_{15}$ &  $w_4 w_{8}$ & $- a_2 - 4$ & $a_3 - 1$ & $a_1+2$\\  
$w_{13}$ & $e$ & $\left\{2 \right\}$ & $3$ & $-2a_1$ & $w_1 w_{13}$ & $w_{13} w_{1}$ & $a_1$ & $-a_2-3$ & $a_3$ \\ 
$w_{14}$ & $(1\, 2)$ & $\left\{2 \right\}$ & $4$ & $2-2a_2$ & $w_1 w_{14}$ & $w_{13} w_{2}$  &$a_2-1$ & $-a_1-4$ &  $a_3$ \\ 
$w_{15}$ & $(1\, 3)$ & $\left\{2 \right\}$ & $4$ & $4-2a_3$ & $w_1 w_{15}$ & $w_{16} w_{5}$ & $a_3-2$ & $-a_2-3$ & $a_1 + 2$\\
$w_{16}$ & $(2\, 3)$ & $\left\{2 \right\}$ & $2$ & $-2a_1$ & $w_1 w_{16}$ & $w_{16} w_{1}$ & $a_1$ & $-a_3-2$ & $a_2 + 1$\\  
$w_{17}$ & $(1\, 2\, 3)$ & $\left\{2 \right\}$ & $5$ & $4-2a_3$ & $w_1 w_{17}$ & $w_{13} w_{5}$& $a_3 - 2$ & $- a_1 -4$ & $a_2+1$ \\ 
$w_{18}$ & $(3\, 2\, 1)$ & $\left\{2 \right\}$ & $3$ & $2-2a_2$ & $w_1 w_{18}$ & $w_{16} w_{2}$  & $a_2 - 1$ & $-a_3 - 2$ & $a_1 + 2$\\ 
$w_{19}$ & $e$ & $\left\{1, 2 \right\}$ & $8$ & $10+2a_1$ & $w_2 w_{20}$  & $w_{13} w_{7}$ & $-a_1-5$ & $-a_2-3$ & $a_3$ \\ 
$w_{20}$ & $(1\, 2)$ & $\left\{1, 2 \right\}$ & $7$ & $8+2a_2$ & $w_1 w_{20}$ & $w_{13} w_{8}$ & $-a_2-4$ & $-a_1 -4$ & $a_3$ \\
$w_{21}$ & $(1\, 3)$ & $\left\{1, 2 \right\}$ & $5$ & $6+2a_3$ & $w_1 w_{21}$ & $w_{16} w_{11}$ & $-a_3-3$ & $- a_2-3$ & $a_1 + 2$\\ 
$w_{22}$ & $(2\, 3)$ & $\left\{1, 2 \right\}$ & $7$ & $10+2a_1$ & $w_2 w_{23}$ & $w_{16} w_{7}$ & $-a_1-5$ & $-a_3-2$ & $a_2 + 1$\\ 
$w_{23}$ & $(1\, 2\, 3)$ & $\left\{1, 2 \right\}$ & $6$ & $6+2a_3$ & $w_1 w_{23}$ &  $w_{13} w_{11}$ & $-a_3 - 3$ & $-a_1-4$ &$a_2 + 1$ \\ 
$w_{24}$ & $(3\, 2\, 1)$ & $\left\{1, 2 \right\}$ & $6$ & $8+2a_2$ & $w_2 w_{21}$ & $w_{16} w_{8}$ & $- a_2 - 4$ & $-a_3 - 2$ & $a_1 + 2$ \\ 

\noalign{\smallskip}\hline

\end{tabular}
\end{table}

\begin{thm} \label{Thm1}
Let $V_\lambda$ be the finite dimensional irreducible representation of $\mathrm{GO}(2, 5)$ with highest weight $\lambda = a_1 \epsilon_1 + a_2 \epsilon_2 + a_3 \epsilon_3 + c \kappa$. One has:
\begin{enumerate}
\item If $a_2 \neq 0$ then there are no ghost classes in the cohomology space $H^\bullet(\partial \overline{S}, \V_\lambda)$.

\item If $a_2 = 0$ (which implies $a_3=0$ and therefore in terms of fundamental weights one has $\lambda = a_1 \varpi_1 + c \kappa$), then the only possible weights in the mixed Hodge structure of the space of ghost classes are the middle weight and the middle weight plus one.
\end{enumerate}
\end{thm}

\begin{proof}

By Lemma~\ref{mwelimination} and  the information in the Table~\ref{table1} one can see that the spaces $Ind_{\mathrm{P}_0}^{\mathrm{G}} H^0(S^{\mathrm{M}_0}, \widetilde{W}_{w_\ast(\lambda)})$ will not contribute to ghost classes for 
\[
w \in \left\{ w_1, w_4, w_6, w_{13}, w_{14}, w_{15}, w_{16}, w_{17}, w_{18} \right\}\,.
\]
On the other hand,  $w_2$ could contribute to ghost classes only if $a_2 = 0$ (which clearly implies $a_3 = 0$). $w_3$ and $w_5$ could contribute to ghost classes only if $a_3 = 0$.

Following Theorem ~\ref{GL2} and similar steps as the ones taken in Theorem~\ref{Thm2},
we continue with analyzing the image of $r_2:H^\bullet(\partial_2, \widetilde{V}_\lambda) \rightarrow H^\bullet(\partial_0, \widetilde{V}_\lambda)$. If $w \in \mathcal{W}^{\mathrm{P}_0}
$ is written as
$w = w^{\mathrm{P}_2/\mathrm{P}_0}w^{\mathrm{P}_2}$ with respect to the decomposition $\mathcal{W}^{\mathrm{P}_0} = \mathcal{W}_2^0 \mathcal{W}^{\mathrm{P}_2}$, then for 
\[
	w \in \left\{ w_5, w_{7}, w_{8}, w_{9}, w_{10}, w_{11}, w_{12}, w_{22} \right\}
\]
one has, $w^{\mathrm{P}_2/\mathrm{P}_0}=w_2 \neq e $, and its component in $\mathcal{W}^{\mathrm{P}_2}$ is, respectively, $w_4, w_{14} , w_{13}, w_{18}, w_{17}, w_{16}, w_{15}$ and $w_{23}$. For each of these $w^{\mathrm{P}_2}$, we can see, following the values of $n_1$ and $n_2$ in the expression $(w^{\mathrm{P}_2})_\ast(\lambda) = n_1\epsilon_1 + n_2\epsilon_2 + n_3\epsilon_3 +c\kappa$ encoded in the last three columns of the Table~\ref{table1}, that $n_1 > n_2$. By Theorem ~\ref{GL2}, this implies that the associated space $Ind_{\mathrm{P}_0}^{\mathrm{G}} H^0(S^{\mathrm{M}_0}, \widetilde{W}_{w_\ast(\lambda)})$ will be entirely contained in the image of $r_2$ and therefore this space cannot contribute to ghost classes. However, for $w_2, w_3, w_{19}, w_{24}$, we made the following observation. For $w_2$ and $w_{19}$, the corresponding space is not entirely contained in the image of $r_2$ only when $a_1 = a_2$ whereas  for $w_3$ and $w_{24}$ this will happen only when $a_2 = a_3$. 

In the case $w_{19}$, one has that the space $Ind_{\mathrm{P}_0}^{\mathrm{G}}  H^{0}(S^{\mathrm{M}_0}, \widetilde{W}_{(w_{19})_\ast(\lambda)})$ could contribute to ghost classes in degree $9$. On the other hand, the symmetric space associated to $\mathrm{G}$ has dimension $10$ and by Corollary $11.4.3$ in \cite{BoSe73}, $H^9(S, \widetilde{V}_\lambda) = 0$. As a conclusion $H^9(S_K, \widetilde{V}_\lambda) \rightarrow H^9(\partial \overline{S}_K, \widetilde{V}_\lambda)$ is the zero morphism and there are no ghost classes in degree $9$ cohomology. Therefore, $w_{19}$ does not contribute to ghost classes.

Therefore the only possible contributions to ghost classes come from the following six Weyl representatives 
\begin{equation}\label{eq:rw25}
	w \in \left\{ w_2, w_{3}, w_{20}, w_{21}, w_{23}, w_{24} \right\} \,. \nonumber
\end{equation}

\par We will study now each one of these cases to determine whether they could actually contribute to ghost classes and in that case what the possible weights in the corresponding mixed Hodge structure are. We do this by studying the image of $r_1:H^\bullet(\partial_{\mathrm{P}_1}, \V_\lambda) \rightarrow H^\bullet(\partial_{\mathrm{P}_0}, \V_\lambda)$ following the discussion of Subsection~\ref{r1}.

One has $s_1=w_{13}$, $w_\ast(\lambda) = n_1 \epsilon_1 + n_2 \epsilon_2 +n_3 \epsilon_3 + c \kappa$, $\rho= \frac{5}{2}\epsilon_1 + \frac{3}{2}\epsilon_2+\frac{1}{2}\epsilon_3$, $\mathrm{A}_{\mathrm{P}_0}^{\mathrm{P}_1} =\{ \epsilon_2\}$, $d_1=3$ and $-w(\lambda+\rho)=-(w_\ast(\lambda)+\rho)$. Then, in this case, $\Lambda_w^{\mathrm{P}_1}= -(n_2+\frac{3}{2})$ and the inequality in item $(a)$ of Theorem ~\ref{EisensteinCohomologyThm} is given by $-n_2 > 3$.

We see that for $w =  w_{20}$, its component in $\mathcal{W}^{\mathrm{P}_1/\mathrm{P_0}}$ with respect to the decompostion $\mathcal{W}^{\mathrm{P_0}}=\mathcal{W}^{\mathrm{P}_1/\mathrm{P_0}} \mathcal{W}^{\mathrm{P}_1}$ is $w_{13}$ ($\ell(w_{13}) = 3 > \frac{d_1}{2}$) and $w$ satisfies the condition $-n_2 > 3$. Thus, a direct application of item $(1)$ of Theorem ~\ref{EisensteinCohomologyThm} gives that $Ind_{\mathrm{P}_0}^{\mathrm{G}} H^0(S^{\mathrm{M}_0}, \widetilde{W}_{w_\ast(\lambda)})$ is contained in $Im(r_1)$ and therefore it does not contribute to ghost classes. By the same procedure, but for $w = w_{23}$, one can see that $Ind_{\mathrm{P}_0}^{\mathrm{G}} H^0(S^{\mathrm{M}_0}, \widetilde{W}_{(w_{23})_\ast(\lambda)})$ is contained in $Im(r_1)$ and therefore it does not contribute to ghost classes. On the other hand, the same calculations for $w_{21}$ show that this element could only contribute to ghost classes if $a_2=0$ (because in that case $-n_2 = a_2+3$).

Finally, we make use of item $(b)$ in Theorem \ref{EisensteinCohomologyThm}. For $w \in \mathcal{W}^{\mathrm{P}_1}$, the highest weight  $w_\ast(\lambda) = n_1 \epsilon_1 + n_2 \epsilon_2 +n_3 \epsilon_3 + c \kappa$  is regular for $\mathrm{M}_1$ if $n_2 > n_3 > 0$.

Assume $a_1 > a_2$. For $w = w_{21} = w_{16}w_{11} \in \mathcal{W}^{\mathrm{P}_1/\mathrm{P_0}} \mathcal{W}^{\mathrm{P}_1}$, one has $\ell(w_{16}) > \frac{d_1}{2}$ and $\Lambda_w^{\mathrm{P}_1} = a_2+\frac{3}{2}$. We will assume $a_2=0$, since we already proved that this element could only contribute to ghost classes in that case. On the other hand, under these assumptions,  $(w_{11})_\ast(\lambda)$ is regular. Therefore, if $a_1 > a_2 = 0$, all the hypothesis of item $(b)$ of the aforementioned theorem are satisfied. Thus, for every form $[\varphi] \in Ind_{\mathrm{P}_0}^{\mathrm{G}} H^0(S^{\mathrm{M}_0}, \widetilde{W}_{w_\ast(\lambda)})$, one has $r_1([E(\varphi, \Lambda_w^{\mathrm{P}_i})]) = [\varphi]$. This implies that $Ind_{\mathrm{P}_0}^{\mathrm{G}} H^0(S^{\mathrm{M}_0}, \widetilde{W}_{w_\ast(\lambda)})$ is contained in $Im(r_1)$. In conclusion, $w_{21}$ could contribute to ghost classes only if $a_1 = 0$. By the same procedure and using the already proved fact that $w_{24}$ can contribute to ghost classes only if $a_2=a_3$, one can show that $w_{24}$ can contribute to ghost classes only if $a_2=0$.

We now summarize the above discussion to point out the possible contribution of the space $Ind_{\mathrm{P}_0}^{\mathrm{G}} H^0(S^{\mathrm{M}_0}, \widetilde{W}_{{w}_\ast(\lambda)})$ for $w\in \mathcal{W}^{\mathrm{P}_0}$ to the ghost classes, as follows:

\begin{enumerate}

\item If $a_1 = a_2 = a_3 =0$ (\ie $V_\lambda$ is one dimensional) then the space $Ind_{\mathrm{P}_0}^{\mathrm{G}} H^0(S^{\mathrm{M}_0}, \widetilde{W}_{{(w_2)}_\ast(\lambda)})$ could contribute to ghost classes in degree $2$ and would have weight equal to the middle weight of $H^2(S, \widetilde{V}_\lambda)$.

\item If $a_2 = a_3 =0$ then the space $Ind_{\mathrm{P}_0}^{\mathrm{G}} H^0(S^{\mathrm{M}_0}, \widetilde{W}_{{(w_3)}_\ast(\lambda)})$ could contribute to ghost classes in degree $4$ and would have weight equal to the middle weight of $H^4(S, \widetilde{V}_\lambda)$.

\item If $a_1=a_2 = a_3 =0$ then the space $Ind_{\mathrm{P}_0}^{\mathrm{G}} H^0(S^{\mathrm{M}_0}, \widetilde{W}_{{(w_{21})}_\ast(\lambda)})$ could contribute to ghost classes in degree $6$ and would have weight equal to the middle weight of $H^6(S, \widetilde{V}_\lambda)$.

\item If $a_2 = a_3 =0$ then the space $Ind_{\mathrm{P}_0}^{\mathrm{G}} H^0(S^{\mathrm{M}_0}, \widetilde{W}_{{(w_{24})}_\ast(\lambda)})$ could contribute to ghost classes in degree $7$ and would have weight equal to the middle weight of $H^7(S, \widetilde{V}_\lambda)$ plus one.

\end{enumerate}

This completes the proof. 
 
\end{proof}

We conclude this section with the following corollary that follows from the proof of Theorem ~\ref{Thm1}.

\begin{coro} \label{Coro3}
Let $V_\lambda$ be the finite dimensional irreducible representation of $GO(2, 5)$ with highest weight $\lambda = n_1 \varpi_1 + n_2 \varpi_2 + n_3 \varpi_3 + c \kappa$. Then ghost classes can exist only if $n_2=n_3=0$, and in that case one has:
\begin{enumerate}
\item If $n_1 \neq 0$ then ghost classes can exist only in degree $4$ with middle weight and in degree $7$ with middle weight plus one.

\item If $n_1 = 0$ then ghost classes can exist only in degrees $2$, $4$ and $6$ with middle weight and in degree $7$ with middle weight plus one.
\end{enumerate}
\end{coro}

\section*{Acknowledgements}

The authors are grateful to the Mathematisches Institut, Georg-August Universit\"at G\"ottingen, and the Max Planck Institute for Mathematics (MPIM) in Bonn where
much of the discussion and work took place, for their support and hospitality. This work has been
initiated while both authors were postdoctoral fellows at MPIM. The authors would like to thank
G\"unter Harder for several discussions and encouragement during the writing of this article. The second author would like to thank Michael Harris and Roberto Miatello for their help and support.

\par Both authors would like to thank the anonymous referee for careful reading and for giving many constructive comments and useful suggestions which substantially helped in improving the presentation of the article. 

\nocite{}
\bibliographystyle{amsalpha}
\bibliography{GC}

\end{document}